\documentclass{amsart}
\usepackage{xcolor}
\definecolor{nicered}{rgb}{0.6, 0, 0.1}
\definecolor{niceblue}{rgb}{0.06, 0.3, 0.57}
\definecolor{nicegreen}{rgb}{0.0, 0.51, 0.5}
\usepackage{color}

\usepackage[colorlinks=true,pagebackref,hyperindex,citecolor=nicegreen,linkcolor=niceblue,urlcolor=nicered]{hyperref}
\usepackage{amsmath,amsfonts,amssymb}
\usepackage{mathrsfs}
\usepackage{mathtools}
\usepackage{microtype}
\usepackage[all]{xy}
\usepackage[T1]{fontenc}
\usepackage{lmodern}
\usepackage{xparse}
\usepackage[shortlabels]{enumitem}
\usepackage{comment}
\setlist[enumerate]{leftmargin=2em,label=\textup{(\roman*)}}

\usepackage{mathbbol}
\usepackage{soul}
\usepackage[french,english]{babel}

\usepackage{cleveref}
\crefname{equation}{Eq.}{Eqs.}
\crefname{theorem}{Theorem}{Theorems} 
\crefname{lemma}{Lemma}{Lemmas}
\crefname{corollary}{Corollary}{Corollaries}
\crefname{proposition}{Proposition}{Propositions}
\crefname{definition}{Definition}{Definitions}
\crefname{remark}{Remark}{Remarks}
\crefname{example}{Example}{Examples}
\crefname{notation}{Notation}{Notations}
\crefname{setup}{Setup}{Setup}
\crefname{question}{Question}{Question}
\crefname{convention}{Convention}{Conventions}

\allowdisplaybreaks

\newtheorem{theorem}{Theorem}[section]
\newtheorem{corollary}[theorem]{Corollary}
\newtheorem{lemma}[theorem]{Lemma}
\newtheorem{proposition}[theorem]{Proposition}

\newtheorem{theoremx}{Theorem}

\theoremstyle{definition}
\newtheorem{definition}[theorem]{Definition}

\newtheorem{example}[theorem]{Example}
\newtheorem{remark}[theorem]{Remark}

\newtheorem{convention}[theorem]{Convention}

\newtheorem{setup}[theorem]{Setup}
\numberwithin{equation}{section}

\newcommand{\m}{\mathfrak{m}}
\newcommand{\p}{\mathfrak{p}}
\newcommand{\n}{\mathfrak{n}}

\newcommand{\N}{\mathfrak{N}}
\newcommand{\A}{\mathfrak{A}}
\newcommand{\ideala}{\mathfrak{a}}

\newcommand{\FF}{\mathbb{F}}
\newcommand{\KK}{\mathbb{K}}

\newcommand{\QQ}{\mathbb{Q}}
\newcommand{\RR}{\mathbb{R}}
\newcommand{\cJ}{\mathcal{J}}

\newcommand{\Cech}{ \check{\rm{C}}}

\newcommand{\ZZ}{\mathbb{Z}}
\newcommand{\NN}{\mathbb{N}}
\newcommand{\PP}{\NN_{>0}} 
\newcommand{\kk}{\mathbb{k}}

\newcommand{\act}{\mathbin{\vcenter{\hbox{\scalebox{0.6}{$\bullet$}}}}}
\newcommand{\club}{\mathbin{\vcenter{\hbox{\scalebox{0.55}{$\clubsuit$}}}}}

\renewcommand{\geq}{\geqslant}
\renewcommand{\leq}{\leqslant}
\renewcommand{\ge}{\geqslant}
\renewcommand{\le}{\leqslant}

\newcommand{\ds}{\operatorname{s}^{\operatorname{diff}}}

\newcommand{\dif}[1]{^{{\langle #1\rangle}}}

\newcommand{\Hom}{\operatorname{Hom}}
\newcommand{\GrHom}[4]{\Hom_{#1}(#2,#3)_{#4}}
\newcommand{\HomogHom}[3]{\GrHom{#1}{#2}{#3}{0}}
\newcommand{\End}{\operatorname{End}}
\newcommand{\GrEnd}[3]{\End_{#1}(#2)_{#3}}
\newcommand{\HomogEnd}[2]{\GrEnd{#1}{#2}{0}}

\newcommand{\ord}{\operatorname{ord}}

\newcommand{\Ass}{\operatorname{Ass}}
\newcommand{\Dim}{\operatorname{Dim}}
\newcommand{\Len}{\operatorname{length}}
	
\newcommand{\id}{\operatorname{id}}

\newcommand{\gr}{\operatorname{gr}}	
\newcommand{\e}{\operatorname{e}}

\newcommand{\ann}{\operatorname{Ann}}

\NewDocumentCommand \pd { m o }
{
\IfNoValueTF {#2}
{ \partial_{#1} }
{ \partial_{#1} \act #2 }
}
\NewDocumentCommand \fpd { m o }
{
\IfNoValueTF {#2}
{ \frac{\partial\ }{\partial #1 } }
{ \frac{\partial #2 }{\partial #1 } }
}

\newcommand{\fs}{\boldsymbol{f^s}}

\newcommand{\gbf}[2]{\mathcal{B}^{#1}_{#2}}

\newcommand{\filtration}[3]{\mathcal{#1}^{#2}_{#3}}
\newcommand{\olfiltration}[3]{\overline{\mathcal{#1}}^{\hskip .05em #2}_{\hskip -.03 em #3}}
\newcommand{\tildefiltration}[3]{\widetilde{\mathcal{#1}}^{\hskip .05em #2}_{\hskip -.03 em #3}}

\newcommand{\Gi}[1]{\filtration{F}{#1}{}} 
\newcommand{\Gii}[1]{\filtration{G}{#1}{}}

\newcommand{\lt}{\bar{\delta}}

\newcommand{\bu}{{}^{\bullet}}
\newcommand{\cga}{commutative graded $\kk$-algebra}
\newcommand{\cat}{\mathbf{C}}
\newcommand{\loccit}{\emph{loc.~cit.}}

\usepackage[textwidth=3.3 cm,textsize=small,shadow
]{todonotes}

\newcommand{\details}[2][]{} 
 
\title[Bernstein's inequality and holonomicity]{Bernstein's inequality and holonomicity for certain singular rings}

\author[]{Josep \`Alvarez Montaner{$^1$}}
\address{Departament de Matem\`atiques  and  Institut de Matem\`atiques de la UPC-BarcelonaTech (IMTech),  Universitat Polit\`ecnica de Catalunya, Av.~Diagonal 647,  08028 Barcelona. Centre de Recerca Matem\`atica (CRM)}
\email{josep.alvarez@upc.edu}

\author[]{Daniel J. Hern\'andez{$^2$}}
\address{Department of Mathematics, University of Kansas, Lawrence, KS 66045, USA}
\email{hernandez@ku.edu}

\author[]{Jack Jeffries{$^3$}}
\address{University of Nebraska-Lincoln, Lincoln, NE~68502, USA}
\email{jack.jeffries@unl.edu}

\author[]{Luis N\'u\~nez-Betancourt${^4}$}
\address{Centro de Investigaci\'on en Matem\'aticas, Guanajuato, Gto., M\'exico}
\email{luisnub@cimat.mx}

\author[]{Pedro Teixeira}
\address{Department of Mathematics, Knox College, Galesburg, IL 61401, USA}
\email{pteixeir@knox.edu}

\author[]{Emily E. Witt${^5}$}
\address{Department of Mathematics, University of Kansas, Lawrence, KS 66045, USA}
\email{witt@ku.edu}

\thanks{{$^1$}Partially supported by grants  PID2019-103849GB-I00 (AEI/10.13039/501100011033) and 2017SGR-932 (AGAUR). Also supported by Salvador de Maradiaga grant (ref.~PRX 19/00405).}

\thanks{{$^2$}Partially supported by the NSF Grants DMS-1304250 and DMS-1600702.}

\thanks{{$^3$}Partially supported by the NSF CAREER Award DMS-2044833.}

\thanks{{$^4$}Partially supported by the CONACYT Grant 284598 and C\'atedras Marcos Moshinsky.}

\thanks{{$^5$}Partially supported by the NSF CAREER Award DMS-1945611.}

\makeatletter
\@namedef{subjclassname@2020}{%
  \textup{2020} Mathematics Subject Classification}
\makeatother


\begin{document}
\maketitle

\begin{abstract}
   In this manuscript, we prove the Bernstein inequality and develop the theory of holonomic $D$-modules for rings of invariants of finite groups in characteristic zero, and for strongly $F$-regular finitely generated graded algebras with FFRT in prime characteristic.
   In each of these cases, the ring itself, its localizations, and its local cohomology modules are holonomic.
   We also show that holonomic $D$-modules, in this context, have finite length and we prove the existence of Bernstein--Sato polynomials in characteristic zero. We obtain these results using a more general version of Bernstein filtrations.
\end{abstract}




\setcounter{tocdepth}{1}
\tableofcontents

\section{Introduction}

Let $R$ be a regular $\kk$-algebra, and $D_{R|\kk}$ its ring of $\kk$-linear differential operators, where $\kk$ is a field of characteristic zero.
A fundamental theorem in the theory of $D_{R|\kk}$-modules is the \emph{Bernstein inequality}, which establishes that the dimension of every finitely generated $D_{R|\kk}$-module is at least $\dim(R)$.
Sato, Kawai, and Kashiwara \cite{SKK} and Malgrange \cite{Malgrange79} proved this result for holomorphic functions.
Gabber \cite{Gabber} gave another proof that works in smooth cases such as polynomial rings and formal power series rings over a field.
Joseph (see \cite[Theorem~9.4.2]{Coutinho}) gave a simple proof for polynomial rings, considering the \emph{Bernstein filtration} on $D_{R|\kk}$ instead of the order filtration.
More recently, Bavula \cite{Bavula} proved the Bernstein inequality for polynomial rings in positive characteristic.

Modules with minimal dimension are called \emph{holonomic}, and are central to the theory of $D_{R|\kk}$-modules since they satisfy rather nice properties.
Holonomic modules have finite length and finite-dimensional de Rham cohomology \cite{KashiwaraRham,VdEdRham}.
Furthermore, holonomic modules play an important role in the Riemann--Hilbert correspondence \cite{MKRH0,ZMRH0,MKRH1,ZMRH2,ZMRH1}. 
In addition, holonomicity can be characterized via homological properties \cite{Bjork79}.

The Bernstein inequality and the study of holonomic modules have generally been limited to regular rings; see below for further discussion.
In this work we prove the Bernstein inequality and develop a theory of holonomic modules for certain singular $\kk$-algebras.
Specifically, we summarize our main results in the following theorem, which is pieced together from various results throughout the text.

\begin{theoremx}[{See \Cref{Cor_invariant_Bernstein,Cor_FFRT_Bernstein}}]\label{MainThm}
   Let $R$ be one of the following.
   \begin{enumerate}[$(a)$]
      \item A ring of invariants of the action of a finite group on a polynomial ring over a field $\kk$ of characteristic zero, or
      \item A strongly $F$-regular finitely generated graded algebra over a perfect field $\kk$ of positive characteristic with finite $F$-representation type \textup(FFRT\textup).
   \end{enumerate}
   Then any nonzero $D_{R|\kk}$-module has dimension at least $\dim(R)$.
   In addition, any nonzero $D_{R|\kk}$-module of dimension $\dim(R)$ and finite multiplicity has finite length in the category of $D_{R|\kk}$-modules.
   In particular, any principal localization $R_f$ and any local cohomology module $H^i_I(R)$ has dimension $\dim(R)$, and therefore has finite length as a $D_{R|\kk}$-module.
\end{theoremx}

First, we build upon the work of Bavula \cite{Bavula} on filtered $\kk$-algebras and filtered modules. 
We consider his definition of dimension, which coincides with the usual Gelfand--Kirillov dimension for finitely generated $\kk$-algebras, although these notions are different in general.
We also consider a version of multiplicity for filtered $\kk$-algebras and modules that is convenient for our treatment of holonomic modules.
A key ingredient in Bavula's work is a technical condition that we call \emph{linear simplicity} (see \Cref{linearly_simple}), which gives a sufficient condition for the Bernstein inequality to hold, and to have a theory of holonomic modules (see \Cref{thm:Bern}).
Our main contribution is a proof in this setting establishing that holonomic modules have finite length (see \Cref{ThmLenHol}). 

Let $R$ be a finitely generated graded $\kk$-algebra, and $D_{R|\kk}$ its ring of $\kk$-linear differential operators.
In order to have a good theory of holonomic $D_{R|\kk}$-modules that extends the case of regular rings, we want $D_{R|\kk}$ to be linearly simple with respect to an appropriate filtration, so that  Bernstein's inequality holds.
Moreover, we want the dimension of $D_{R|\kk}$ to be twice the dimension of $R$, and its multiplicity to be finite and positive.
If $R$ satisfies all these properties, we call it a \emph{Bernstein algebra}.
In this case, $R$ is a holonomic $D_{R|\kk}$-module, and the class of holonomic $D_{R|\kk}$-modules is closed under localization and taking local cohomology.
We also show that a certain module over $R\otimes_\kk \kk(s)$ is holonomic, and as a consequence, we obtain the existence of the Bernstein--Sato polynomial for Bernstein algebras in characteristic zero (see \Cref{CorBS1,CorBS2}).

In \Cref{Bernstein_filtration}, we give a generalization of the Bernstein filtration that works for the graded ring $D_{R|\kk}$.
The difficult part when dealing with specific cases is to prove that $D_{R|\kk}$ is linearly simple with respect to the generalized Bernstein filtration.
Sufficient conditions to ensure the remaining properties of a Bernstein algebra are given in terms of \emph{differential signature} \cite{BJNB} (see \Cref{thm:dim-Bernstein-filtration}).
 
We note that in characteristic zero, there are graded hypersurfaces with rational singularities that are not Bernstein algebras, and fail the conclusions of \Cref{MainThm}; see \Cref{eg51,eg52}.
Likewise, in positive characteristic, there are graded hypersurfaces that are strongly $F$-regular or have FFRT (but not both) that are not Bernstein algebras, and again fail the conclusions of \Cref{MainThm}; see \Cref{eg61,eg62}.
Thus, strong hypotheses are necessary for a statement akin to \Cref{MainThm}.

The main results of this work provide classes of singular rings that are Bernstein algebras.
First, we consider the case of a ring of invariants of a finite group $G$ acting linearly on a polynomial ring $R$ over a field $\kk$ of characteristic zero (see \Cref{Cor_invariant_Bernstein}). 
Moreover, we prove in \Cref{holonomic_direct_summand} that a $D_{R^G|\kk}$-module is holonomic if and only if it is a \emph{differential direct summand} of a holonomic $D_{R|\kk}$-module \cite{AMHNB, square}.

For our second class of Bernstein algebras, we focus on algebras over a field of positive characteristic. 
Specifically, we work on rings with finite $F$-representation type---FFRT for short---that were introduced by Smith and Van den Bergh \cite{SVDB}. 
Examples of rings with FFRT include complete regular rings, quotients of polynomial rings by monomial ideals, normal monoid rings, affine cones of Grassmannians, and graded direct summands of polynomial rings.
In \Cref{Cor_FFRT_Bernstein}, we show that certain rings with FFRT are Bernstein algebras.

We point out that Gabber \cite{Gabber} proved the integrability of the characteristic variety for the case of filtered rings whose associated graded rings are commutative and Noetherian.
In this setting, van den Essen \cite{VdE86} considered the notion of holonomicity, but to the best of our knowledge, this notion has only been applied to smooth algebras when dealing with rings of differential operators. We show in \Cref{example:noVDE} that 
for our first class of Bernstein algebras, although Gabber's theorem applies, there are no nonzero $D$-modules that are holonomic in the sense of van den Essen. Thus our results do not follow from Gabber's theorem.  More recently, Losev \cite{Losev} showed a version of the Bernstein inequality
using representation theory of filtered algebras with commutative Noetherian associated graded rings, whose spectra have finitely many symplectic leaves. In particular, this theory applies to certain algebras in characteristic zero, but to the best of our knowledge, 
it does not provide examples of $D_{R|\kk}$-modules satisfying the Bernstein inequality for a singular $\kk$-algebra $R$.

In one of the first uses of $D$-modules in commutative algebra, Lyubeznik \cite{LyuDMod} proved the finiteness of the set of associated primes of a local cohomology module of a regular ring in characteristic zero.
A key point of his argument is that local cohomology modules of holonomic $D$-modules are themselves holonomic, and thus have finite length as $D$-modules. 
The finiteness of sets of associated primes of local cohomology modules has since been established for rings of invariants of finite groups \cite{NBDS} and rings with FFRT \cite{TT,HNB,DQ}.
These proofs, however, did not use the theory of $D$-modules, though recently a new proof using $D$-module theory was given for rings of invariants  \cite{AMHNB}.
\Cref{MainThm} gives another proof for the previously mentioned classes of rings via $D$-modules that resembles the one given by Lyubeznik.

\subsection*{Acknowledgements}
This project was initiated and developed during an American Institute of Mathematics SQuaRE (Structured Quartet Research Ensemble) titled \emph{Roots of Bernstein--Sato polynomials}. 
We are grateful for AIM's hospitality, and for their support of this project. 

\section{Filtrations} \label{SecBackgrpoundFil}

In this section, we establish a number of key facts on filtrations and their numerical invariants that we use throughout this work.
Much, if not all, of the material in this section is known \cite{Bavula}, with the possible exception of the discussion of multiplicity.
To keep the paper self-contained and to avoid unnecessary finiteness hypotheses, we provide proofs for the claims made in this section.

Throughout this manuscript, $\kk$ always denotes a field, which has arbitrary characteristic unless otherwise stated.
Within \Cref{SecBackgrpoundFil,SecBernsteinHol}, $A$ denotes an associative but not necessarily commutative $\kk$-algebra.

\subsection{Filtrations, dimension, and multiplicity}

\begin{definition}\label{notation:filtration}
   Throughout this paper, by a \emph{filtration} $\filtration{F}{\bullet}{}$ on a $\kk$-algebra $A$ we mean an ascending, exhaustive filtration by finite-dimensional $\kk$-vector spaces, indexed by the nonnegative integers, and such that $\filtration{F}{0}{}=\kk$ and $\filtration{F}{i}{}\filtration{F}{j}{} \subseteq \filtration{F}{i+j}{}$ for each $i,j \in \NN$.
   If $\filtration{F}{\bullet}{}$ is a filtration on $A$, we say that $(A,\filtration{F}{\bullet}{})$ is a \emph{filtered $\kk$-algebra}.

   Given a filtered $\kk$-algebra $(A,\filtration{F}{\bullet}{})$, and a left (respectively, right) $A$-module $M$, a \emph{filtration $\filtration{G}{\bullet}{}$ on $M$ compatible with $\filtration{F}{\bullet}{}$} is an ascending, exhaustive filtration on $M$ by finite-dimensional $\kk$-vector spaces, indexed by the nonnegative integers, and such that $\filtration{F}{i}{}\filtration{G}{j}{} \subseteq \filtration{G}{i+j}{}$ (respectively, $\filtration{G}{j}{}\filtration{F}{i}{} \subseteq \filtration{G}{i+j}{}$) for every $i$ and $j$ in $\NN$.
   If $\filtration{G}{\bullet}{}$ is a filtration on $M$ compatible with $\filtration{F}{\bullet}{}$, we say that \emph{$(M,\filtration{G}{\bullet}{})$ is an $(A,\filtration{F}{\bullet}{})$-module}.
\end{definition}

\begin{convention}
   If $\filtration{F}{\bullet}{}$ is a filtration on an algebra or module, we adopt the convention that $\filtration{F}{i}{} = 0$ for each negative integer $i$.
\end{convention}

\begin{definition} \label{dim_mult}
	Let $\filtration{G}{\bullet}{}$ be an ascending sequence of finite-dimensional $\kk$-vector spaces. We define
\begin{align*} \Dim(\filtration{G}{\bullet}{}) &= \inf \left\{ t\in \RR_{\geq 0} \ \big|  \lim\limits_{i\to\infty} \frac{\dim_\kk \filtration{G}{i}{}}{i^t}=0\right\} 
	\\&=\inf \left\{ t\in \RR_{\geq 0} \ \big|  \dim_\kk \filtration{G}{i}{}\leq i^t\;  \ \forall \ i\gg 0 \right\}
	\\ &= \inf \left\{ t\in \RR_{\geq 0} \ \big|  \ \exists \ C: \dim_\kk \filtration{G}{i}{}\leq C i^t\; \ \forall \ i\gg 0 \right\}.
\end{align*}
If $d=\Dim(\filtration{G}{\bullet}{})$ is finite, then the \emph{multiplicity} of $\filtration{G}{\bullet}{}$ is the extended real number
\[
   \e(\filtration{G}{\bullet}{})= \limsup\limits_{i\to\infty} \frac{\dim_\kk \filtration{G}{i}{}}{i^d} \in\RR_{\geq 0} \cup \{ \infty\}.
\]
By convention, if $\Dim(\filtration{G}{\bullet}{})$ is infinite, then we set $\e(\filtration{G}{\bullet}{})=\infty$.

If $(A,\filtration{G}{\bullet}{})$ is a filtered $\kk$-algebra or $(M,\filtration{G}{\bullet}{})$ is a left or right $(A,\filtration{F}{\bullet}{})$-module, then we define the dimension (respectively, multiplicity) of $(A,\filtration{G}{\bullet}{})$, or of $(M,\filtration{G}{\bullet}{})$, as the dimension (respectively, multiplicity) of $\filtration{G}{\bullet}{}$.
\end{definition}

We note that the dimension of a filtered $\kk$-algebra or module may be infinite, or may fail to be an integer, or even rational, if finite (see \Cref{prop:any-larger-dim}).
Furthermore, there are examples showing that the limit superior in the definition of multiplicity cannot simply be replaced with  a limit, as the limit may fail to exist (see \cite[Section~4]{TwoExamples}). 

\begin{definition}\label{DefEquivalenceFiltrations}
Let $\Gi{\bullet}$ and $\Gii{\bullet}$ be two filtrations on a $\kk$-algebra or module.
	\begin{enumerate}
 		\item\label{item:shiftdom} We say that $\Gi{\bullet}$ is \emph{shift dominated} by $\Gii{\bullet}$ if there exists $j\in\PP$ such that 
		\[\Gi{i} \subseteq \Gii{i+j} \quad 		\text{for all} \ i\in \NN.\]
		\item\label{item:lindom} We say that $\Gi{\bullet}$ is \emph{linearly dominated} by $\Gii{\bullet}$ if there exists $C\in\PP$ such that 
		\[\Gi{i} \subseteq \Gii{Ci} \quad \text{for all} \ i\in \NN.\]
		\item We say that $\Gi{\bullet}$ and $\Gii{\bullet}$ are \emph{shift equivalent} if both $\Gi{\bullet}$ is shift dominated by $\Gii{\bullet}$, and $\Gii{\bullet}$ is shift dominated by $\Gi{\bullet}$.
		\item We say that $\Gi{\bullet}$ and $\Gii{\bullet}$ are \emph{linearly equivalent} if both $\Gi{\bullet}$ is linearly dominated by $\Gii{\bullet}$, and $\Gii{\bullet}$ is linearly dominated by $\Gi{\bullet}$.
	\end{enumerate}
\end{definition}

Note that if $\Gi{0} \subseteq \Gii{0}$ (e.g., when $\Gi{\bullet}$ and $\Gii{\bullet}$ are filtrations on a $\kk$-algebra) and there exist positive integers $C$ and $j$ such that $\Gi{i} \subseteq \Gii{Ci+j}$ for all $i \in \PP$, then $\Gi{\bullet}$ is linearly dominated by $\Gii{\bullet}$.
In particular, when $\Gi{0} \subseteq \Gii{0}$, shift domination implies linear domination.
We also note that in condition \ref{item:shiftdom} in \Cref{DefEquivalenceFiltrations}, we may replace ``for all $i\in \NN$'' with ``for all $i\gg 0$.''
The same replacement can be made in condition \ref{item:lindom}, provided that $\Gi{0} \subseteq \Gii{0}$.

Observe also that shift and linear domination are transitive, and that shift and linear equivalence are equivalence relations.

\begin{proposition}\label{prop:lin-dom}
   Let $\Gi{\bullet}$ and $\Gii{\bullet}$ be filtrations on a $\kk$-algebra or module.
   If $\Gi{\bullet}$ is linearly dominated by $\Gii{\bullet}$, then
	\begin{enumerate}
		\item\label{part:lin-dom1} $\Dim(\Gi{\bullet})\leq \Dim(\Gii{\bullet})$.
			\end{enumerate}
			If $\Gi{\bullet}$ is linearly dominated by $\Gii{\bullet}$ and $\Dim(\Gi{\bullet})= \Dim(\Gii{\bullet})$, then
		\begin{enumerate}\setcounter{enumi}{1}
		\item\label{part:lin-dom1.5} $\e(\Gii{\bullet})<\infty$ implies $\e(\Gi{\bullet})<\infty$, and
		\item\label{part:lin-dom2} $\e(\Gii{\bullet})=0$ implies $\e(\Gi{\bullet})=0$.
			\end{enumerate}
Finally, if $\Gi{\bullet}$ is shift dominated by $\Gii{\bullet}$ and $\Dim(\Gi{\bullet})= \Dim(\Gii{\bullet})$, then 
\begin{enumerate}\setcounter{enumi}{3}
\item\label{part:lin-dom3} $\e(\Gi{\bullet})\leq \e(\Gii{\bullet})$.
	\end{enumerate}
\end{proposition}

\begin{proof}
Fix $C\in \PP$ such that $\Gi{i}\subseteq \Gii{Ci}$ for all $i\in \NN$.
		Given $t\in\RR_{\geq 0}$, we have
		\begin{align*}\limsup\limits_{i\to \infty} \frac{\dim_\kk \Gi{i}}{i^t}
		&\leq \limsup\limits_{i\to \infty} \frac{\dim_\kk \Gii{Ci}}{i^t} \\&= C^t \limsup\limits_{i\to \infty} \frac{\dim_\kk \Gii{Ci}}{(Ci)^t} 
                  \leq C^t \limsup\limits_{i\to \infty} \frac{\dim_\kk \Gii{i}}{i^t}.
                \end{align*}
	Parts~\ref{part:lin-dom1}, \ref{part:lin-dom1.5}, and~\ref{part:lin-dom2} then follow from the definitions.
		
For part~\ref{part:lin-dom3}, if $\Gi{\bullet}$ and $\Gii{\bullet}$ have dimension $d$ and finite multiplicity, and $j \in \PP$ is such that $\Gi{i}\subseteq \Gii{i+j}$ for all $i\in \NN$, then 
\begin{align*} \e(\Gi{\bullet}) &= \limsup\limits_{i\to \infty} \frac{\dim_\kk \Gi{i}}{i^d}
\leq \limsup\limits_{i\to \infty} \frac{\dim_\kk \Gii{i+j}}{i^d} \\&= \limsup\limits_{i\to \infty} \frac{(i+j)^d}{i^d} \frac{\dim_\kk \Gii{i+j}}{(i+j)^d} = \limsup\limits_{i\to \infty} \frac{\dim_\kk \Gii{i}}{i^d}= \e(\Gii{\bullet}).\qedhere
\end{align*}
	\end{proof}

The following corollary follows immediately from \Cref{prop:lin-dom}.

\begin{corollary}\label{prop:lin-equiv}
   Let $\Gi{\bullet}$ and $\Gii{\bullet}$ be two filtrations on a $\kk$-algebra or module.
   If $\Gi{\bullet}$ and $\Gii{\bullet}$ are linearly equivalent, then
	\begin{enumerate}
		\item $\Dim(\Gi{\bullet})=\Dim(\Gii{\bullet})$,
		\item $\e(\Gi{\bullet})<\infty$ if and only if $\e(\Gii{\bullet})<\infty$, and
		\item $\e(\Gi{\bullet})>0$ if and only if $\e(\Gii{\bullet})>0$.
	\end{enumerate}
If $\Gi{\bullet}$ and $\Gii{\bullet}$ are shift equivalent, then in addition we have
	\begin{enumerate}\setcounter{enumi}{3}
\item $\e(\Gi{\bullet})=\e(\Gii{\bullet})$. \qed
\end{enumerate}
\end{corollary}

We note that the dimension of a filtration on an algebra (or module) necessarily depends on the choice of filtration, rather than solely on the algebra (or module).

\begin{proposition}\label{prop:any-larger-dim}
   Let $\filtration{F}{\bullet}{}$ be a filtration on a $\kk$-algebra $A$, or on a module $M$, with $\Dim(\filtration{F}{\bullet}{})=d\in\RR_{>0}$.
   Then, for every real number $\lambda\geq d$ there exists a filtration $\filtration{G}{\bullet}{}$ on $A$, or on $M$, such that $\Dim(\filtration{G}{\bullet}{})=\lambda$.
\end{proposition}
\begin{proof}
   Fix $\lambda \geq d$; set $s=\lambda/d \ge 1$ and define $\filtration{G}{i}{}=\filtration{F}{\lfloor i^{s} \rfloor}{}$.
   It is apparent that $\filtration{G}{\bullet}{}$ is ascending, exhaustive, and finite dimensional, and $\filtration{G}{0}{} = \filtration{F}{0}{}$, which equals $\kk$ in the algebra case.
   The fact that $\filtration{G}{\bullet}{}$ is compatible with multiplication follows from the inequality $\lfloor i^{s} \rfloor + \lfloor j^{s} \rfloor \leq \lfloor (i+j)^{s} \rfloor$ in the algebra case, or from the inequality $i+\lfloor j^{s} \rfloor \leq \lfloor (i+j)^{s} \rfloor$ in the module case.

   To compute the dimension of $\filtration{G}{\bullet}{}$, note first that for every $t\ge 0$, if $\dim_{\kk} \filtration{F}{i}{} \leq i^t$ for all $i\gg 0$, then $\dim_{\kk} \filtration{G}{i}{} = \dim_{\kk} \filtration{F}{\lfloor i^{s} \rfloor}{} \leq i^{st}$ for all $i\gg 0$.
   Consequently, $\Dim(\filtration{G}{\bullet}{})\leq s\Dim(\filtration{F}{\bullet}{}) = \lambda$.
   For the reverse inequality, it suffices to show that for any $t<d$ and any $N$, there exists $j>N$ such that $\dim_{\kk} \filtration{G}{j}{} > j^{st}$.
   Fix such $t$ and $N$.
   As $t < d = \Dim(\filtration{F}{\bullet}{})$, we can find $i>N^{s}$ with $\dim_{\kk} \filtration{F}{i}{} > i^{t}(1+1/N)^{st}$. Let $j\in \NN$ be such that $(j-1)^{s} \leq i < j^{s}$ so that, in particular, $j> N$.
   We then have
	\[ \dim_{\kk} \filtration{G}{j}{} \geq \dim_\kk \filtration{F}{i}{} > i^t (1+1/N)^{st} \geq (j-1)^{st}(1+1/N)^{st}\geq j^{st}\]
	as required.
	\end{proof}

\subsection{Finitely generated modules}

\begin{definition}
   Let $(A,\filtration{F}{\bullet}{})$ be a filtered $\kk$-algebra and $M$ a finitely generated left (or right) $A$-module.
   We say that a filtration $\filtration{G}{\bullet}{}$ on $M$ is \emph{standard} if there exists a generating set $v_1,\ldots,v_\ell$ of $M$ such that $\filtration{G}{i}{}=\filtration{F}{i}{}\{v_1,\ldots,v_\ell\}$ (or $\filtration{G}{i}{}=\{v_1,\ldots,v_\ell\}\filtration{F}{i}{}$) for all $i\in \NN$.
   A filtration $\filtration{G}{\bullet}{}$ on $M$ is a \emph{good filtration} if it is shift equivalent to a standard filtration.
\end{definition}

We caution the reader that the term ``standard filtration'' is used for other related notions in the literature; see \Cref{def:algebra-std-filtration}.
We also point out that the notions of standard filtration and of good filtration on a module are dependent on the choice of the filtration $\filtration{F}{\bullet}{}$ on $A$.

\renewcommand{\Gi}[1]{\filtration{G}{#1}{}} 
\renewcommand{\Gii}[1]{\filtration{H}{#1}{}}
\begin{proposition}\label{prop:compare-good}
   Let $(A,\filtration{F}{\bullet}{})$ be a filtered $\kk$-algebra, and $M$ a finitely generated left \textup(or right\textup) $A$-module.
   Let $\Gi{\bullet}$ and $\Gii{\bullet}$ be filtrations on $M$ compatible with~$\filtration{F}{\bullet}{}$.
   If $\Gi{\bullet}$ is a good filtration, then $\Gi{\bullet}$ is shift dominated by $\Gii{\bullet}$, and consequently $\Dim(\Gi{\bullet})\leq \Dim(\Gii{\bullet})$, and $\e(\Gi{\bullet})\leq \e(\Gii{\bullet})$ whenever these dimensions agree.
   Consequently, if both $\Gi{\bullet}$ and $\Gii{\bullet}$ are good filtrations, then $\Gi{\bullet}$ and $\Gii{\bullet}$ are shift equivalent, $\Dim(\Gi{\bullet})=\Dim(\Gii{\bullet})$, and $\e(\Gi{\bullet})= \e(\Gii{\bullet})$.
\end{proposition}

\begin{proof}
   Without loss of generality, take $M$ to be a left $A$-module.
   Since $\Gi{\bullet}$ is shift equivalent, and hence shift dominated, by a standard filtration, transitivity of shift domination allows us to assume that $\Gi{\bullet}$ is standard; say $\Gi{\bullet}=\filtration{F}{\bullet}{}\{v_1,\ldots,v_\ell\}$.
   If $j\in\PP$ is such that $\{v_1,\ldots,v_\ell\}\subseteq \Gii{j}$, then we have 
   \[\Gi{i}=\filtration{F}{i}{}\{v_1,\ldots,v_\ell\} \subseteq \filtration{F}{i}{} \Gii{j} \subseteq \Gii{i+j}\]
   for each $i\in \NN$, so $\Gi{\bullet}$ is shift dominated by $\Gii{\bullet}$.
   The remaining claims now follow from \Cref{prop:lin-dom,prop:lin-equiv}.
\end{proof}

\begin{definition}\label{def:module-mult-alg-filt}
   Let $(A,\filtration{F}{\bullet}{})$ be a filtered $\kk$-algebra, and $M$ a finitely generated $A$-module.
   We define $\Dim(M,\filtration{F}{\bullet}{})\coloneqq \Dim(\filtration{G}{\bullet}{})$ and $\e(M,\filtration{F}{\bullet}{})\coloneqq \e(\filtration{G}{\bullet}{})$ for a good filtration $\filtration{G}{\bullet}{}$ of $M$ compatible with $\filtration{F}{\bullet}{}$. 
\end{definition}

Note that by \Cref{prop:compare-good}, the dimension and multiplicity of a finitely generated module over a filtered $\kk$-algebra do not depend on the choice of the good filtration~$\filtration{G}{\bullet}{}$.
Our next result shows that both the dimension, and the positivity and finiteness of the multiplicity depend only on the linear equivalence class of the filtration on~$A$.

\begin{proposition}\label{prop:lin-equiv-alg-mod}
   Let $A$ be a $\kk$-algebra with filtrations $\filtration{F}{\bullet}{1}$ and $\filtration{F}{\bullet}{2}$, and $M$ a finitely generated left \textup(or right\textup) $A$-module.
   Let $\filtration{G}{\bullet}{1}$ and $\filtration{G}{\bullet}{2}$ be good filtrations on $M$ compatible, respectively, with $\filtration{F}{\bullet}{1}$ and $\filtration{F}{\bullet}{2}$.
   If $\filtration{F}{\bullet}{1}$ is linearly dominated by $\filtration{F}{\bullet}{2}$, then $\filtration{G}{\bullet}{1}$ is linearly dominated by $\filtration{G}{\bullet}{2}$.
   Consequently, $\Dim(\filtration{G}{\bullet}{1})\leq \Dim(\filtration{G}{\bullet}{2})$, and when equality holds, if $\e(\filtration{G}{\bullet}{2})<\infty$ then $\e(\filtration{G}{\bullet}{1})<\infty$, and if $\e(\filtration{G}{\bullet}{2})=0$ then $\e(\filtration{G}{\bullet}{1})=0$.
\end{proposition}

\begin{proof}
   \Cref{prop:compare-good} allows us to replace $\filtration{G}{\bullet}{1}$ and $\filtration{G}{\bullet}{2}$ by any good filtrations.
   Thus, we may assume that $\filtration{G}{\bullet}{1}$ and $\filtration{G}{\bullet}{2}$ are standard filtrations corresponding to the same generating set $\{v_1,\dots,v_\ell\}$ of $M$, in which case the first claim follows.
   The remaining claims follow at once from \Cref{prop:lin-dom}.
\end{proof}

\subsection{Finitely generated algebras}

\begin{definition}\label{def:algebra-std-filtration}
   Let $A$ be a finitely generated $\kk$-algebra.
   We say that a filtration $\filtration{F}{\bullet}{}$ on $A$ is \emph{standard} if there exists a generating set $v_1,\ldots,v_\ell$ of $A$ such that $\filtration{F}{1}{}=\kk\{v_1,\ldots,v_\ell\}$ and for each $i\in \PP$, $\filtration{F}{i}{}=(\filtration{F}{1}{})^i$, the $\kk$-subspace generated by monomials in $v_1,\ldots,v_\ell$ of degree $\le i$.
\end{definition}

The following proposition is the analogue of \Cref{prop:compare-good} for finitely generated $\kk$-algebras, where the role of shift domination is now played by linear domination.

\renewcommand{\Gi}[1]{\filtration{F}{#1}{}} 
\renewcommand{\Gii}[1]{\filtration{G}{#1}{}}
\begin{proposition}\label{prop:algebra-std-filtrations}
   Let $A$ be a finitely generated $\kk$-algebra, with filtrations $\Gi{\bullet}$ and $\Gii{\bullet}$.
   If $\Gi{\bullet}$ is a standard filtration, then $\Gi{\bullet}$ is linearly dominated by $\Gii{\bullet}$, and thus $\Dim(\Gi{\bullet})\leq \Dim(\Gii{\bullet})$.
   If both $\Gi{\bullet}$ and $\Gii{\bullet}$ are standard filtrations, then $\Gi{\bullet}$ and $\Gii{\bullet}$ are linearly equivalent, and consequently, $\Dim(\Gi{\bullet})=\Dim(\Gii{\bullet})$, $\e(\Gi{\bullet})<\infty$ if and only if $\e(\Gii{\bullet})<\infty$, and $\e(\Gi{\bullet})>0$ if and only if $\e(\Gii{\bullet})>0$.
\end{proposition}

\begin{proof}
   For the first claim, we fix generators $v_1,\ldots,v_\ell$ such that $\Gi{1}=\kk\{v_1,\ldots,v_\ell\}$ and $\Gi{i}=(\Gi{1})^i$ for all $i\in \PP$.
   Let $C\in \PP$ be such that $\{v_1,\dots,v_\ell\} \subseteq \Gii{C}$, so $\Gi{1}\subseteq \Gii{C}$.
   Then for each $i\in \PP$ we have 
   \[\Gi{i} =(\Gi{1})^i \subseteq (\Gii{C})^i \subseteq \Gii{Ci}.\]
   The remaining claims now follow from \Cref{prop:lin-dom,prop:lin-equiv}.
\end{proof}

Let $A$ be a finitely generated $\kk$-algebra, and $\filtration{F}{\bullet}{}$ a standard filtration on $A$.
The Gelfand--Kirillov dimension of $A$ is $\operatorname{GK}(A)\coloneqq \Dim(\filtration{F}{\bullet}{})$, which does not depend on the choice of the standard filtration $\filtration{F}{\bullet}{}$ by \Cref{prop:algebra-std-filtrations}.
For a non-finitely generated $\kk$-algebra~$A$, the Gelfand--Kirillov dimension is defined as 
\[ \operatorname{GK}(A) \coloneqq \sup \left\{ \operatorname{GK}(A') \ \big| \ A'\text{ is a finitely generated $\kk$-subalgebra of }A \right\}. \]

\subsection{Finitely generated commutative associated graded algebras}

Consider a filtered $\kk$-algebra $(A,\filtration{F}{\bullet}{})$, and let $\gr(\filtration{F}{\bullet}{})=\bigoplus_{i\geq 0} \filtration{F}{i}{}/\filtration{F}{i-1}{}$ be its associated graded ring.
Given an $(A,\filtration{F}{\bullet}{})$-module $(M,\filtration{G}{\bullet}{})$, the associated graded module $\gr(\filtration{G}{\bullet}{})=\bigoplus_{i\geq 0} \filtration{G}{i}{}/\filtration{G}{i-1}{}$ is a $\gr(\filtration{F}{\bullet}{})$-module.

In this subsection, we focus on the case where $\gr(\filtration{F}{\bullet}{})$ is a finitely generated commutative $\kk$-algebra.

\begin{proposition}\label{prop:good-filtration First Noeth case}
   Let $(A,\filtration{F}{\bullet}{})$ be a filtered $\kk$-algebra.
   Suppose that the associated graded ring $\gr(\filtration{F}{\bullet}{})$ is a finitely generated commutative $\kk$-algebra.
   Let $M$ be a left or right $A$-module.
   A filtration $\filtration{G}{\bullet}{}$ on $M$ compatible with $\filtration{F}{\bullet}{}$ is a good filtration if and only if $\gr(\filtration{G}{\bullet}{})$ is a finitely generated $\gr(\filtration{F}{\bullet}{})$-module.
   Furthermore, any lift of a generating set for $\gr(\filtration{G}{\bullet}{})$ is a generating set for $M$.
\end{proposition}

\begin{proof}
   Without loss of generality, we suppose $M$ is a left $A$-module.
   Assume first that $\filtration{G}{\bullet}{}$ is a standard filtration on $M$, that is, $\filtration{G}{\bullet}{}=\filtration{F}{\bullet}{}\{v_1,\ldots,v_\ell\}$ for some generating set $v_1,\ldots,v_\ell$ of $M$.
   Let $\overline{u} \in \filtration{G}{i}{}/ \filtration{G}{i-1}{}$ be a nonzero homogeneous element of $\gr(\filtration{G}{\bullet}{})$, which is the class of $u\in \filtration{G}{i}{} \smallsetminus \filtration{G}{i-1}{}$.
   Since $u = \sum_{j=1}^\ell a_jv_j $ for some $a_j\in \filtration{F}{i}{}$, the associated graded module $\gr(\filtration{G}{\bullet}{})$ is generated as a $\gr(\filtration{F}{\bullet}{})$-module by the classes of the $v_j$.
   Any good filtration $\tildefiltration{G}{\bullet}{}$ is shift equivalent to a standard filtration, so using \cite[Lemma~1.12]{VdE86} we conclude that $\gr(\tildefiltration{G}{\bullet}{})$ is also a finitely generated $\gr(\filtration{F}{\bullet}{})$-module.

   Conversely, suppose $\gr(\filtration{G}{\bullet}{})$ is generated by homogeneous elements $\overline{u}_1, \dots, \overline{u}_s$ as a $\gr(\filtration{F}{\bullet}{})$-module, and let $u_j\in \filtration{G}{n_j}{}\smallsetminus \filtration{G}{n_j-1}{}$ be a lift of $\overline{u}_j$, for each $j$.
   Thus, any nonzero homogeneous element $\overline{u}\in \filtration{G}{i}{}/ \filtration{G}{i-1}{}$ can be written in the form $\overline{u}= \sum_{j=1}^s \overline{a_j}\overline{u_j}$ with $a_j\in \filtration{F}{i-n_j}{}$, and
   \[u-(a_1u_1 + \cdots +a_s u_s) \eqqcolon x^{(1)} \in \filtration{G}{i-1}{}.\]
   Applying the same argument to $x^{(1)}$, we can find $a_j^{(1)}\in \filtration{F}{i-n_j-1}{}$, for $j=1,\dots s$, such that $x^{(1)} - (a_1^{(1)}u_1 + \cdots +a_s^{(1)} u_s) \eqqcolon x^{(2)} \in \filtration{G}{i-2}{}$, and thus
   \[u=\big(a_1+a_1^{(1)}\big) u_1 + \cdots + \big(a_s + a_s^{(1)}\big) u_s + x^{(2)}.\]
   Repeating this process we eventually end up with some $x^{(j)}=0$, and thus
   \[\filtration{G}{i}{} = \filtration{F}{i- n_1}{} u_1 + \cdots + \filtration{F}{i- n_s}{} u_s\]
   showing that $\filtration{G}{\bullet}{}$ is shift equivalent to the standard filtration $\filtration{F}{\bullet}{}\{u_1,\ldots,u_s\}$.
   This construction also shows that any lift of a generating set for $\gr(\filtration{G}{\bullet}{})$ is a generating set for $M$.
\end{proof}

\begin{proposition}
   Let $(A,\filtration{F}{\bullet}{})$ be a filtered $\kk$-algebra.
   Suppose that the associated graded ring $\gr(\filtration{F}{\bullet}{})$ is a finitely generated commutative $\kk$-algebra.
   If $M$ is a finitely generated $A$-module, then $\Dim(M,\filtration{F}{\bullet}{})$ is an integer, and $\e(M,\filtration{F}{\bullet}{})$ is a positive rational number.
\end{proposition}
     
\begin{proof}
   Let $\filtration{G}{\bullet}{}$ be a good filtration for $M$.
   Then $\gr(\filtration{G}{\bullet}{})$ is a finitely generated graded $\gr(\filtration{F}{\bullet}{})$-module by \Cref{prop:good-filtration First Noeth case}.
   By the theory of Hilbert functions on commutative finitely generated graded $\kk$-algebras, $\dim_{\kk} \filtration{G}{i}{} = \sum_{j=0}^i \dim_{\kk}(\gr(\filtration{G}{\bullet}{})_j)$ agrees with a quasipolynomial function of $i$ with rational coefficients, for $i \gg 0$, and both claims follow.
\end{proof}

\section{Bernstein's inequality and holonomic modules} \label{SecBernsteinHol}

In this section we recall the Bernstein inequality for what we call \emph{linearly simple algebras}.
The ideas behind this class of algebras can be found in Bavula's work \cite[Theorem~3.1]{Bavula}.
We rephrase his work for our purposes in the first part of \Cref{thm:Bern}, and include a proof for the convenience of the reader.
The second part of \Cref{thm:Bern} concerns multiplicities of filtered modules, which was not treated in Bavula's work in the generality needed in this manuscript. 

\begin{definition} \label{linearly_simple}
   A filtered $\kk$-algebra $(A,\filtration{F}{\bullet}{})$ is \emph{$C$-linearly simple} for some $C\in\PP$ if for each $i\in \NN$ and each $\delta \in \filtration{F}{i}{} \smallsetminus \{0\}$,
   \[ 1 \in \filtration{F}{Ci}{} \, \delta \, \filtration{F}{Ci}{}. \]
   We say that $(A,\filtration{F}{\bullet}{})$ is \emph{linearly simple} if it is $C$-linearly simple for some $C\in\PP$.
\end{definition}

\begin{remark}
   \label{rem: linear simplicity only depends on linear equivalence class}
   Linear simplicity of a filtered $\kk$-algebra depends only on the linear equivalence class of the filtration.
   Indeed, suppose $\filtration{F}{\bullet}{}$ and $\filtration{G}{\bullet}{}$ are linearly equivalent filtrations on a $\kk$-algebra $A$, with $\filtration{F}{i}{} \subseteq \filtration{G}{Ki}{}$ and $\filtration{G}{i}{} \subseteq \filtration{F}{Li}{}$ for each $i\in \NN$.
   If $(A,\filtration{F}{\bullet}{})$ is $C$-linearly simple, then one easily verifies that $(A,\filtration{G}{\bullet}{})$ is $KLC$-linearly simple. 
\end{remark}

The following key result implicitly appears in Bavula's work \cite[Proof of Theorem~3.1]{Bavula}.
We single it out because we use it to show new properties for holonomic modules in the generality we need.

\begin{lemma}\label{LemmaHomBernstein}
   Let $(A,\filtration{F}{\bullet}{})$ be a filtered $\kk$-algebra, and $(M,\filtration{G}{\bullet}{})$ a left $(A,\filtration{F}{\bullet}{})$-module.
   Suppose that $(A,\filtration{F}{\bullet}{})$ is $C$-linearly simple.
   Let $\Psi: \filtration{F}{i}{} \to \Hom_\kk(\filtration{G}{(C+1)i}{},\filtration{G}{(C+2)i}{})$ be defined by $\delta \mapsto \psi_\delta$, where $ \psi_\delta(v)=\delta v$.
   If $\filtration{G}{i}{}\neq 0$, then $\Psi$ is injective.
   The analogous result holds for right modules as well.
\end{lemma}

\begin{proof}
We prove the contrapositive:
Suppose that there exists $\delta \in \filtration{F}{i}{} \smallsetminus \{0\}$ such that $ \psi_\delta=0$.
Then $\delta \filtration{F}{Ci}{} \filtration{G}{i}{}\subseteq \delta \filtration{G}{(C+1)i}{} = 0$, and thus $\delta \filtration{F}{Ci}{} \filtration{G}{i}{}=0$.
Since $1\in \filtration{F}{Ci}{} \delta \filtration{F}{Ci}{}$, we have $\filtration{G}{i}{} \subseteq \filtration{F}{Ci}{} \delta \filtration{F}{Ci}{} \filtration{G}{i}{} = 0$, and therefore $\filtration{G}{i}{} =0$.
\end{proof}

\begin{theorem}[Bernstein Inequality {\cite[Theorem~3.1]{Bavula}}]\label{thm:Bern}
   Let $(A,\filtration{F}{\bullet}{})$ be a filtered $\kk$-algebra with $\Dim(\filtration{F}{\bullet}{})<\infty$, and $(M,\filtration{G}{\bullet}{})$ a nontrivial left or right $(A,\filtration{F}{\bullet}{})$-module.
   Suppose that $(A,\filtration{F}{\bullet}{})$ is $C$-linearly simple.
   Then,
   \[\Dim(\filtration{G}{\bullet}{}) \geq \frac{1}{2}\Dim(\filtration{F}{\bullet}{}).\]
   Moreover, if $\theta\coloneqq \Dim(\filtration{G}{\bullet}{})=\frac{1}{2}\Dim(\filtration{F}{\bullet}{})$, then
   \[\e(\filtration{G}{\bullet}{})\geq \frac{\sqrt{ \e (\filtration{F}{\bullet}{}) }}{ (C+1)^{\theta/2}(C+2)^{\theta/2}}.\]
\end{theorem}
     
\begin{proof}
           The first claim holds  if $\Dim(\filtration{G}{\bullet}{})$ is infinite, so suppose that is not the case.
           Let $t > \theta \coloneqq \Dim(\filtration{G}{\bullet}{})$.
           For all sufficiently large $i$ we have $\dim_\kk \filtration{G}{i}{}\le i^t$ by the definition of dimension, as well as $\filtration{G}{i}{} \ne 0$ by the nontriviality of $M$.
           \Cref{LemmaHomBernstein} then shows that, for such $i$,
	\[
	\dim_\kk \filtration{F}{i}{}\leq ((C+1)i)^{t}( (C+2)i)^{t}
	= (C+1)^t (C+2)^{t} i^{2t},
	\]
	and it follows that $\Dim(\filtration{F}{\bullet}{})\leq 2t$.
	Since this holds for all $t > \theta$, we conclude that $\frac{1}{2}\Dim(\filtration{F}{\bullet}{}) \leq \theta = \Dim(\filtration{G}{\bullet}{})$.
        
	Now assume that $\theta=\Dim(\filtration{G}{\bullet}{})=\frac{1}{2}\Dim(\filtration{F}{\bullet}{})$.
        Invoking \Cref{LemmaHomBernstein} one more time, we have 
	\begin{align*}\e (\filtration{F}{\bullet}{})&= \limsup\limits_{i\to \infty} \frac{\dim_\kk{\filtration{F}{i}{}}}{ i^{2\theta}} \leq \limsup\limits_{i\to \infty} \frac{\dim_\kk{\filtration{G}{(C+1)i}{}} \cdot \dim_\kk{\filtration{G}{(C+2)i}{}}}{ i^{2\theta}}\\
	&\leq \left( \limsup\limits_{i\to \infty} \frac{(C+1)^{\theta} \dim_\kk{\filtration{G}{(C+1)i}{}}}{((C+1)i)^{\theta}} \right)
	\left(\limsup\limits_{i\to \infty} \frac{(C+2)^{\theta} \dim_\kk{\filtration{G}{(C+2)i}{}}}{((C+2)i)^{\theta}}\right)\\
	&\leq (C+1)^{\theta}(C+2)^{\theta}\left( \limsup\limits_{i\to \infty} \frac{\dim_\kk \filtration{G}{i}{}}{i^{\theta}}\right)^2 
	= (C+1)^{\theta}(C+2)^{\theta}  \e (\filtration{G}{\bullet}{})^2,
	\end{align*}
	and the claimed inequality follows.
\end{proof}

\begin{definition}
   Let $(A,\filtration{F}{\bullet}{})$ be a linearly simple filtered $\kk$-algebra such that $\dim(\filtration{F}{\bullet}{})<\infty$ and $0<\e(\filtration{F}{\bullet}{})<\infty$.
A nonzero $A$-module is \emph{holonomic} if it admits a filtration $\filtration{G}{\bullet}{}$ of dimension $\frac{1}{2}\Dim(\filtration{F}{\bullet}{})$ and with finite multiplicity; the zero module is also holonomic by convention.
\end{definition}

We shall see in \Cref{ThmLenHol} that a holonomic module has finite length, so in particular it must be finitely generated.  
We also point out that there are $A$-modules that admit filtrations of dimension $\frac{1}{2}\Dim(\filtration{F}{\bullet}{})$ but are not holonomic \cite[p.~224]{Bavula}.

\begin{proposition}\label{PropSubQuotSum}
If $(A,\filtration{F}{\bullet}{})$ is a linearly simple filtered $\kk$-algebra with finite dimension, and positive and finite multiplicity, then the following hold\textup:
\begin{enumerate}
   \item \label{item: holonomicity is preserved by submodules and quotients}
   Every submodule and quotient of a holonomic $A$-module is holonomic.
   \item \label{item: holonomicity is preserved by direct sums}
   Every finite direct sum of holonomic $A$-modules is holonomic.
\end{enumerate}
\end{proposition}

\begin{proof}
	\begin{enumerate}
           \item Let $M$ be a nonzero holonomic $A$-module, and let $\filtration{G}{\bullet}{}$ be a filtration on $M$ of dimension $\frac{1}{2}\Dim(\filtration{F}{\bullet}{})$ and finite multiplicity.
           For a nonzero proper submodule $N$ of $M$, $N\cap \filtration{G}{\bullet}{}$ is a filtration on $N$ with $\dim_\kk (N\cap \filtration{G}{i}{}) \leq \dim_\kk \filtration{G}{i}{}$, so $\Dim(N\cap \filtration{G}{\bullet}{})\leq \Dim(\filtration{G}{\bullet}{}) = \frac{1}{2}\Dim(\filtration{F}{\bullet}{})$, and if equality holds, then the multiplicity is finite.
           But equality holds by \Cref{thm:Bern}, so $N$ is holonomic.           
           We show that the quotient $M/N$ is holonomic in similar fashion, using the filtration $(\filtration{G}{\bullet}{}+N)/N$. 
           \item This reduces to the case of two modules, say $M_1$ and $M_2$, with filtrations $\filtration{G}{\bullet}{1}$ and $\filtration{G}{\bullet}{2}$ of dimension $\frac{1}{2}\Dim(\filtration{F}{\bullet}{})$ and finite multiplicity. Then $\filtration{G}{\bullet}{1}\oplus \filtration{G}{\bullet}{2}$ is a filtration on $M_1\oplus M_2$ of dimension $\frac{1}{2}\Dim(\filtration{F}{\bullet}{})$ and finite multiplicity. \qedhere
	\end{enumerate}
\end{proof}

\begin{remark}\label{HolonomicNotPrefect}
   We recall that an extension of two holonomic $A$-modules might not be holonomic \cite[Section~3]{TwoExamples}.
   This implies that multiplicity is not additive, nor subadditive.
   Furthermore, even for a holonomic $A$-module with a standard filtration, the limit superior in the definition of multiplicity cannot be changed to a plain limit, as the sequence in the definition may fail to converge \cite[Section~4]{TwoExamples}.
\end{remark}

It is known that a holonomic module over the ring of differential operators on a polynomial ring has finite length \cite[Theorem~9.6]{Bavula}.
We now show this for a more general class of algebras.  

\begin{theorem}\label{ThmLenHol}
Let $(A,\filtration{F}{\bullet}{})$ be a $C$-linearly simple filtered $\kk$-algebra with finite dimension and positive finite multiplicity.
If $M$ is a holonomic $A$-module, then $M$ has finite length as an $A$-module. Furthermore,
\[\Len_A M\leq 	\frac{\e(\filtration{G}{\bullet}{})^2(C+1)^\theta (C+2)^\theta}{ \e(\filtration{F}{\bullet}{})}\]
for any filtration $\filtration{G}{\bullet}{}$ on $M$ of dimension $\theta \coloneqq \frac{1}{2} \Dim(\filtration{F}{\bullet}{})$ and finite multiplicity.
\end{theorem}

\begin{proof}
   Let $0=M_0\subsetneqq M_1\subsetneqq \cdots \subsetneqq M_t=M$ be a chain of submodules of $M$, which we may assume is a nontrivial $A$-module.
   Given a filtration $\filtration{G}{\bullet}{}$ on $M$ of dimension $\theta \coloneqq \frac{1}{2} \Dim(\filtration{F}{\bullet}{})$ and finite multiplicity, let $\olfiltration{G}{\bullet}{j}$ be the filtration on $M_j/M_{j-1}$ given by $\olfiltration{G}{i}{j} = (\filtration{G}{i}{}\cap M_j+M_{j-1})/M_{j-1}$.
   We note that, because $\olfiltration{G}{i}{j} \cong (\filtration{G}{i}{}\cap M_j)/(\filtration{G}{i}{}\cap M_{j-1})$, the sum $\sum^t_{j=1} \dim_\kk \olfiltration{G}{i}{j}$ telescopes to $\dim_\kk \filtration{G}{i}{}$ for each $i$.
   By \Cref{LemmaHomBernstein},
   \[\dim_\kk \filtration{F}{i}{}\leq \dim_\kk \olfiltration{G}{(C+1)i}{j} \dim_\kk \olfiltration{G}{(C+2)i}{j}\]
   for every $j$ and all sufficiently large $i$.
   Adding up over $j = 1,\ldots,t$, we obtain
   \begin{align*}
  t\dim_\kk \filtration{F}{i}{} &\le \sum_{j=1}^t \Big(\dim_\kk \olfiltration{G}{(C+1)i}{j} \dim_\kk \olfiltration{G}{(C+2)i}{j} \Big)\\
                &\le \bigg(\sum_{j=1}^t \dim_\kk \olfiltration{G}{(C+1)i}{j} \bigg)\bigg(\sum_{j=1}^t\dim_\kk \olfiltration{G}{(C+2)i}{j}\bigg)
  =\dim_\kk \filtration{G}{(C+1)i}{}  \dim_\kk \filtration{G}{(C+2)i}{}.
   \end{align*}
   Dividing by $i^{2\theta}$ and taking limit superior yields
   \begin{align*}
  t \e(\filtration{F}{\bullet}{}) &\le \limsup_{i\to\infty}\frac{\dim_\kk \filtration{G}{(C+1)i}{}  \dim_\kk \filtration{G}{(C+2)i}{}}{i^{2\theta}}\\
  &\le \bigg(\limsup_{i\to\infty}\frac{\dim_\kk \filtration{G}{(C+1)i}{}}{i^{\theta}}\bigg)
    \bigg(\limsup_{i\to\infty}\frac{\dim_\kk \filtration{G}{(C+2)i}{}}{i^{\theta}}\bigg)\\
  &= (C+1)^\theta (C+2)^\theta \bigg(\limsup_{i\to\infty}\frac{\dim_\kk \filtration{G}{(C+1)i}{}}{((C+1)i)^{\theta}}\bigg)
    \bigg(\limsup_{i\to\infty}\frac{\dim_\kk \filtration{G}{(C+2)i}{}}{((C+2)i)^{\theta}}\bigg)\\
  &\le (C+1)^\theta (C+2)^\theta \e(\filtration{G}{\bullet}{})^2.
   \end{align*}
   We conclude that
   \[t\leq \frac{\e(\filtration{G}{\bullet}{})^2(C+1)^\theta (C+2)^\theta}{ \e(\filtration{F}{\bullet}{})}\]
   which proves our two claims.
\end{proof}

\begin{remark}
   \Cref{ThmLenHol} shows in particular that any holonomic module is finitely generated, so the notion of good filtration applies, and adopting the notation of that theorem, we have
	\[\Len_A M\leq 	\frac{\e(M,\filtration{F}{\bullet}{})^2(C+1)^\theta (C+2)^\theta}{ \e(\filtration{F}{\bullet}{})}.\]
\end{remark}

\begin{proposition}
   Let $(A,\filtration{F}{\bullet}{})$ be a linearly simple filtered $\kk$-algebra with nonzero finite dimension and multiplicity.
   Then any holonomic $A$-module is cyclic.
\end{proposition}

\begin{proof}
   This follows essentially from the classic proof for the Weyl algebra \cite[Theorem~10.2.5]{Coutinho}.
   Note that the same proof works in this context without the left Noetherian hypothesis on the ring, as $M$ has finite length and $A$ is not an Artinian module over itself.
\end{proof}

\section{Rings of differential operators and Bernstein algebras} \label{Sec4}

From this section onward, we direct our focus to rings of differential operators on commutative $\kk$-algebras.
After recalling some basic terminology, we take a look at local cohomology modules from the standpoint of $D$-module theory, and use the machinery developed in the previous sections to determine sufficient conditions under which local cohomology modules are holonomic---which in particular implies the finiteness of the sets of associated primes of these modules.
Narrowing our focus to rings of differential operators on commutative finitely generated graded $\kk$-algebras, we introduce a class of filtrations on those rings that generalize the classic Bernstein filtration on the Weyl algebra.
We conclude by introducing a class of algebras over which one can develop a nice theory of holonomic $D$-modules.

\subsection{Rings of differential operators}

\subsubsection*{Generalities}

Let $R$ be a commutative $\kk$-algebra, and consider the ring of $\kk$-linear endomorphisms $\End_{\kk}(R)$.
The $\kk$-linear differential operators of order $\le i$, where $i$ is a nonnegative integer, are defined inductively as follows:
A differential operator of order $0$ is simply the multiplication by an element of $R$.
If $i>0$, then a differential operator of order $\le i$ is a $\kk$-linear map $\delta:R\to R$ such that for
every $r\in R$, the commutator $[\delta, r] \coloneqq \delta\circ r - r\circ \delta$ is a differential operator of order $\le i-1$, where we consider $ r:R\to R$ as the multiplication by $r$.
Equivalently, a $\kk$-linear map $\delta:R\to R$ is a differential operator of order $\le i$ if for every $r_0,\ldots,r_i \in R$ the $(i+1)$-fold commutator $[\cdots [[\delta,r_0],r_1],\ldots, r_i]$ is zero.
We say that $\delta \in \End_{\kk}(R)$ is a differential operator of order $i$, and write $\ord(\delta) = i$, if $\delta$ is a differential operator of order $\le i$, but not of order $< i$. 

\begin{remark}
   \label{rmk: iterated commutators}
   For later use, we note that if $R$ is generated as a $\kk$-algebra by a set~$\Sigma$, then to verify that $\delta \in \End_\kk(R)$ is a differential operator of order $\le i$, it suffices to verify that $[\delta,r]$ is a differential operator of order $\le i-1$ for all $r\in \Sigma$, or that $[\cdots [[\delta,r_0],r_1],\ldots, r_i] = 0$ for all $r_0,\ldots,r_i\in \Sigma$.
\end{remark}

The set consisting of all $\kk$-linear differential operators on $R$ of order~$\le i$ is a $\kk$-subspace of $\End_\kk(R)$, which we denote by $D^i_{R|\kk}$.
Differential operators of all orders form a ring 
\[ D_{R|\kk} \coloneqq \bigcup_{i\in \NN} D^i_{R|\kk} \subseteq {\End}_{\kk}(R).\]
The chain of $\kk$-vector spaces $D^0 _{R|\kk} \subseteq D^1_{R|\kk} \subseteq D^2_{R|\kk} \subseteq \cdots$ is called the \emph{order filtration} on $D_{R|\kk}$, though we caution the reader that this is not a filtration in the sense of \Cref{notation:filtration}, since these are generally not finite-dimensional $\kk$-vector spaces and $D^0 _{R|\kk} \cong R \not\cong \kk$.
Despite that, we occasionally extend some terminology introduced for filtrations to the order filtration---specifically, we say that a filtration $\filtration{F}{\bullet}{}$ on $D_{R|\kk}$ is linearly dominated by the order filtration if there exists $C\in \PP$ such that $\filtration{F}{i}{} \subseteq D_{R|\kk}^{Ci}$ for every $i\in \NN$.



\subsubsection*{Differential operators in positive characteristic}

Turning to positive characteristic, suppose now that $R$ is a commutative algebra over a perfect field $\kk$ of characteristic $p>0$.
Assume that $R$ is $F$-finite, that is, $R$ is finitely generated as an $R^{p^e}$-module for some (equivalently, all) $e>0$, where $R^{p^e}$ is the subring of $R$ consisting of all $p^e$-th powers of its elements.
Then 
\[D_{R|\kk} = \bigcup_{e\in\NN}D^{(e)}_{R|\kk}\]
where $D^{(e)}_{R|\kk}\coloneqq \End_{R^{p^e}}(R)$ consists of the differential operators of \emph{level $e$}
\cite[Theorem~2.7]{SmithSP} (see also \cite[Theorem~1.4.9]{yekutieli.explicit_construction}).

If $R$ is reduced, we may consider the ring $R^{1/p^e}=\{r^{1/p^e} \; | \; r\in R\}$ consisting of the $p^e$-th roots of the elements of $R$, and identify $D^{(e)}_{R|\kk}$ with $\End_{R}(R^{1/p^e})$ via the map
$\delta\mapsto \delta^{1/p^e}$, where
\begin{equation}
   \delta^{1/p^e}(r^{1/p^e})=\delta(r)^{1/p^e}.\label{eq: definition of roots of maps}
\end{equation}
We note that $D_{R|\kk}^{p^e-1}\subseteq D_{R|\kk}^{(e)}$, and so
\begin{equation}
   \label{eq: comparing order and level filtrations}
   D^i_{R|\kk}\subseteq D_{R|\kk}^{(\lceil \log_p(i+1)\rceil)}
\end{equation}
and if $R$ is generated by $n$ elements as a $\kk$-algebra, then
\begin{equation}\label{eq: comparing order and level filtrations 2}
   D_{R|\kk}^{(e)} \subseteq D^{n(p^e-1)}_{R|\kk}.
\end{equation}
Both of these facts are established in the work of Smith \cite[Theorem~2.7]{SmithSP}.

\subsubsection*{Differential operators on polynomial rings}
If $S = \kk[x_1,\ldots,x_n]$ is a polynomial ring over a field $\kk$ of characteristic zero, then $D_{S|\kk}$ coincides with the Weyl algebra $S\langle\partial_1,\dots, \partial_n\rangle$, where $\partial_i =\frac{\partial\ \,}{\partial x_i}$ are the partial derivatives.
A $\kk$-linear differential operator on $S$ can be written, in its unique \emph{normal form}, as $\delta= \sum_{\alpha,\beta} a_{\alpha\beta} x^\alpha \partial^\beta$, where $\alpha,\beta \in \mathbb{N}^n$, all but finitely many of the coefficients $a_{\alpha\beta}\in \kk$ are zero, and we adopt the multi-index notation $x^\alpha\coloneqq x_1^{\alpha_1}\cdots x_n^{\alpha_n}$, $\partial^\beta\coloneqq \partial_1^{\beta_1} \cdots \partial_n^{\beta_n}$.

The order filtration in this case is given by
\[D^i_{S|\kk}=\bigg\{ \sum_{\alpha,\beta} a_{\alpha\beta} x^\alpha \partial^\beta : |\beta| \leq i \bigg\},\]
where $|\beta|=\beta_1+\cdots+\beta_n$.
Note that the $D^i_{S|\kk}$ are not finite-dimensional $\kk$-vector spaces. 
The main example of a finite-dimensional filtration on $D_{S|\kk}$ is the \emph{Bernstein filtration}  $\gbf{\bullet}{S}$ given by
\[\gbf{i}{S}=\bigg\{ \sum_{\alpha,\beta} a_{\alpha\beta} x^\alpha \partial^\beta : |\alpha| +|\beta| \leq i \bigg\},\]
which we shall study in greater generality in \Cref{sec: gbf}.

\begin{remark}\label{rem: weights}
   In the terminology used by Smith  \cite{GregSmith} and Boldoni \cite{Boldini}, the order filtration and the Bernstein filtration are filtrations associated to the weight vectors $(\mathbf{0},\mathbf{1})\in \mathbb{Z}^{2n}$ and $(\mathbf{1},\mathbf{1})\in \mathbb{Z}^{2n}$, respectively.
   More generally one may consider filtrations associated to weight vectors $(\underline{u},\underline{v})\in \mathbb{Z}^{2n}$.
\end{remark}

If $S=\kk[x_1,\dots, x_n]$ is a polynomial ring with coefficients in a perfect field $\kk$ of characteristic $p>0$, then 
\[D_{S|\kk}=S \left \langle \frac{1}{p^e !} \partial_i ^{p^e}  \ \big| \   i=1,\dots , n , \  e \in \NN \right \rangle\]
where $\frac{1}{p^e !} \partial_i ^{p^e}$ is the operator that maps $x_1^{\alpha_1}\cdots x_n^{\alpha_n} \mapsto \binom{\alpha_i}{p^e}x_1^{\alpha_1}\cdots x_i^{\alpha_i-p^e}\cdots x_n^{\alpha_n}$ (so, in particular, $x_i^{p^e} \mapsto 1$). 

\subsubsection*{Differential operators on finitely generated algebras}

Returning to arbitrary characteristic, we now consider differential operators on finitely generated algebras.
Let $S$ be a polynomial ring over a field $\kk$ of arbitrary characteristic, and $R = S/ I$ for some ideal $I\subseteq S$.
The ring of $\kk$-linear differential operators on $R$ has been described in terms of the $\kk$-linear differential operators on~$S$ \cite[Theorem~15.5.13]{McC_Rob} (see also \cite{Milicic,Moncada}).
Namely, we have
\begin{equation}\label{opsonquot}
   D_{R|\kk} \cong \frac{ \{ \delta\in D_{S|\kk} : \delta (I)\subseteq I\}}{I D_{S|\kk}}.
\end{equation}
The order of the differential operators is preserved under the previous isomorphism; thus the order filtration on $D_{R|\kk}$ is given by
\[D^i_{R|\kk} \cong \frac{ \{ \delta\in D^i_{S|\kk} \ | \  \delta (I) \subseteq I\}}{I D^i_{S|\kk}}.\]

\subsubsection*{Differential operators on graded algebras}

When discussing commutative graded algebras, we shall adopt the following convention.
\begin{convention}
   \label{conv: commutative graded ring}
   Throughout this paper, by a \emph{\cga} we mean a positively graded commutative ring $R= \bigoplus_{i\in \NN} R_i$ with $R_0 =\kk$, that is finitely generated as a $\kk$-algebra.
\end{convention}

If $R$ is a \cga, the ring of differential operators $D_{R|\kk}$ naturally inherits a $\ZZ$-grading that extends the grading on $R$, where we declare a differential operator $\delta: R \to R$ to be homogeneous of degree $d$ if $\delta(R_i) \subseteq R_{i+d}$ for each $i$.

If $S$ is a standard graded polynomial ring, then under the previous grading, the partial derivatives $\partial_i$ are homogeneous of degree $-1$.
More generally, if we consider $S$ as a graded ring with $\deg(x_i)=w_i$, then 
 $D_{S|\kk} $  is a graded ring with $\deg(\partial_i)=-w_i$, and $\deg\big( \frac{1}{p^e !} \partial_i ^{p^e} \big) = -p^e w_i$ when $\kk$ is a field of positive characteristic~$p$.
 As a consequence, if $\delta \in D^i_{S|\kk}$, then $\deg(\delta)\geq -iw$, where $w=\max\{w_i\}$.
 By \eqref{opsonquot}, if $I$ is a homogeneous ideal of $S$ and $R=S/I$, then $ D_{R|\kk} $ is a graded algebra, again satisfying $\deg(\delta)\geq -iw$ for every $\delta \in D^i_{R|\kk}$.

\subsection{\v{C}ech and local cohomology as $D$-modules}

A commutative $\kk$-algebra $R$ has a natural structure of a left $D_{R|\kk}$-module.
In general, given a $D_{R|\kk}$-module $M$ and an element $f\in R$, the localization $M_f$ is also a left $D_{R|\kk}$-module.
We define the action of a differential operator $\delta$ of order zero by $\delta ( \frac{v}{f^t} )=\frac{ \delta (v)}{f^t}$.
Assuming the action has been defined for differential operators of order less than $n$, for $\delta\in D^{n}_{R|\kk}$ we define
\[\delta \left(\frac{v}{f^t}\right) = \frac{\delta (v) - [\delta,f^t] (\frac{v}{f^t})}{f^t}.\]
Note that this well defined, since $[\delta,f^t]$ has order at most $n-1$.
With this $D_{R|\kk}$-module structure on $M_f$, the
localization map $M\to M_f$ is a morphism of $D_{R|\kk}$-modules.

The \v{C}ech complex of $M$ with respect to a sequence of elements $\underline{f}=f_1,\ldots,f_\ell\in R$ is defined by
\[
\Cech^\bullet(\underline{f};M): \hskip 3mm 0\to M\to \bigoplus_i
M_{f_i}\to\bigoplus_{i,j} M_{f_i f_j}\to \cdots \to M_{f_1 \cdots
f_\ell} \to 0
\]
where the maps on every summand are localization maps up to a sign.
The \v{C}ech cohomology modules of $M$ with respect to the sequence $\underline{f}$ are defined by
\[
H^j_{\underline{f}}(M)=H^j(\Cech^\bullet(\underline{f};M)).
\]
If $\underline{g}$ is another sequence of elements of $R$ such that $(\underline{f}) = (\underline{g}) \eqqcolon I$, then $H^j_{\underline{f}}(M)=H^j_{\underline{g}}(M)$ for each $j$, which justifies our denoting this module simply by $H^j_I(M)$.

The \v{C}ech cohomology modules $H^j_I(M)$ inherit a $D_{R|\kk}$-module structure from their construction, and agree with the local cohomology modules of $M$ with support in $I$ whenever $R$ is a Noetherian ring.

We show in this section that holonomicity is preserved by localization and \v{C}ech cohomology.

\begin{lemma}\label{commute-f}
      Let $R$ be a commutative $\kk$-algebra. Let $f\in R$ and $\delta\in D_{R|\kk}$.
   Set $\delta^{(0)}=\delta$, and for each positive integer $i$  define $\delta^{(i)} \in D_{R|\kk}$ inductively by $\delta^{(i)}=[\delta^{(i-1)},f]$.
   Then for each $j \in \ZZ$ we have the following identities in $D_{R_f|\kk}$\textup:
	\[  \delta f^{j}=\sum_{i=0}^{\ord(\delta)}   \binom{j}{i} f^{j-i}  \delta^{(i)} \qquad  \text{and} \qquad {f}^j \delta =\sum_{i=0}^{\ord(\delta)} (-1)^{i} \binom{j}{i} \delta^{(i)} f^{j-i},  \]
	where $\binom{j}{i}=\frac{j\cdot (j-1)\cdots(j-i+1)}{i!}$.	
\end{lemma}

\begin{proof}
   The proofs of the two identities are similar, so we prove only the first one.
   We first prove it for $j \ge 0$ by induction on $j$, observing that the identity holds trivially for $j=0$.
   For $j > 0$, using our induction hypothesis on $\delta f^{j-1}$ we obtain
   \begin{align*}
     \delta f^{j}
     &= \delta f^{j-1} f = \sum_{i\ge 0}  \binom{j-1}{i} f^{j-1-i} \delta^{(i)} f \\
     &= \sum_{i\ge 0}  \binom{j-1}{i} f^{j-1-i} \big(f \delta^{(i)}  + \delta^{(i+1)}\big)  \\
     &=  \sum_{i\ge 0}  \binom{j-1}{i} f^{j-i} \delta^{(i)}  + \sum_{i\ge 0}  \binom{j-1}{i} f^{j-1-i} \delta^{(i+1)} \\
     &=  \sum_{i\ge 0}  \binom{j-1}{i} f^{j-i} \delta^{(i)}  + \sum_{i\ge 0}  \binom{j-1}{i-1} f^{j-i} \delta^{(i)} \\
     &= \sum_{i\ge 0}  \binom{j}{i}  f^{j-i} \delta^{(i)}.
   \end{align*}
   We now prove the first identity for $j\le 0$, using induction on $\ord(\delta) - j$, observing that the identity holds trivially when $j=0$ or $\ord(\delta)=0$.
   If $j<0$, we have
   \[\delta f^j = f^{-1}f \delta f^j = f^{-1}\big(\delta f - \delta^{(1)}\big) f^j = f^{-1}\delta f^{j+1} - f^{-1}\delta^{(1)}f^j,\]
   and applying our induction hypothesis to both $\delta f^{j+1}$ and $\delta^{(1)}f^j$ we see that
   \begin{align*}
     \delta f^j
     &= f^{-1} \cdot\sum_{i\ge 0} \binom{j+1}{i} f^{j+1-i} \delta^{(i)} - f^{-1}\cdot\sum_{i\ge 0} \binom{j}{i} f^{j-i}\delta^{(i+1)} \\
     &= \sum_{i\ge 0} \binom{j+1}{i} f^{j-i} \delta^{(i)} - \sum_{i\ge 0} \binom{j}{i-1} f^{j-i}\delta^{(i)} \\
     &= \sum_{i\ge 0} \binom{j}{i} f^{j-i} \delta^{(i)}. \qedhere 
   \end{align*}
\end{proof}

\begin{lemma}\label{LemmaLocSameDim}
   Let $R$ be a commutative $\kk$-algebra.
   Let $\filtration{F}{\bullet}{}$ be a filtration on $D_{R|\kk}$ that is linearly dominated by the order filtration.
   Let $M$ be a left $D_{R|\kk}$-module with a filtration $\filtration{G}{\bullet}{}$ that is compatible with $\filtration{F}{\bullet}{}$.
   Suppose that $\filtration{G}{\bullet}{}$ has finite dimension $\theta$ and finite multiplicity.
   Then for any $f\in R$, there exists a filtration $\tildefiltration{G}{\bullet}{}$ on $M_f$ that is compatible with $\filtration{F}{\bullet}{}$, has dimension at most $\theta$, and if its dimension equals $\theta$, then its multiplicity is finite.
\end{lemma}

\begin{proof}
   Given $f\in R$, choose $a$ such that $f\in \filtration{F}{a}{}$; fix $C$ such that $\filtration{F}{i}{} \subseteq D_{R|\kk}^{Ci}$ for all~$i$.
   Set $\tildefiltration{G}{j}{} \coloneqq \frac{1}{f^{Cj}} \filtration{G}{j(Ca+1)}{}$.
   Then $\tildefiltration{G}{\bullet}{}$ is a finite-dimensional, ascending, exhaustive filtration on $M_f$, and the claims about the dimension and multiplicity of $\tildefiltration{G}{\bu}{}$ follow from the fact that $\dim_\kk \tildefiltration{G}{j}{} \le \dim_\kk \filtration{G}{j(Ca+1)}{}$ for each $j$.
  
   We need to verify that this filtration is compatible with $\filtration{F}{\bullet}{}$.
   Using the notation of \Cref{commute-f}, if the order of $\delta$ is less than or equal to $t$, we get
   \[ f^{j+t} \delta f^{-j} = \delta f^t - \binom{j+t}{1} \delta^{(1)} f^{t-1} + \cdots + (-1)^t \binom{j+t}{t} \delta^{(t)}  \]
   as an equality in $D_{R_f|\kk}$; in particular, this operator restricts to an operator in $D_{R|\kk}$.
   If $\delta\in \filtration{F}{i}{}$, then $\delta$ has order at most $Ci$ and $\delta^{(k)}\in \filtration{F}{i+ak}{}$ for all $k$, so the previous equation shows that $f^{Cj+Ci} \delta f^{-Cj} \in \filtration{F}{i(Ca+1)}{}.$
   We then have
   \begin{align*} \delta \cdot \tildefiltration{G}{j}{} &= \delta \cdot \frac{1}{f^{Cj}} \filtration{G}{j(Ca+1)}{} 
                                                   \subseteq \frac{1}{f^{Cj+Ci}} \filtration{F}{i(Ca+1)}{} \filtration{G}{j(Ca+1)}{} \\&\subseteq \frac{1}{f^{C(i+j)}} \filtration{G}{(i+j)(Ca+1)}{} = \tildefiltration{G}{i+j}{}
   \end{align*}
   as required.
\end{proof}

\begin{theorem}\label{ThmLocCohHol}
   Let $R$ be a commutative $\kk$-algebra.
   Let $\filtration{F}{\bullet}{}$ be a filtration on $D_{R|\kk}$ that is linearly dominated by the order filtration.
   Suppose that $(D_{R|\kk},\filtration{F}{\bullet}{})$ is linearly simple with finite dimension and finite positive multiplicity.
   If $M$ is a holonomic $D_{R|\kk}$-module, then the following hold.
   \begin{enumerate}
      \item $M_f$ is holonomic for every $f\in R$.
      \item $H^j_I(M)$ is holonomic for every finitely generated ideal $I\subseteq R$ and $j\in\NN$.
   \end{enumerate}
\end{theorem}

\begin{proof}
   The claim for a localization of $M$ follows immediately from \Cref{LemmaLocSameDim}.
   Finite direct sums of localizations of $M$ are therefore holonomic $D_{R|\kk}$-modules by \Cref{PropSubQuotSum}\ref{item: holonomicity is preserved by direct sums}.
   Thus, the kernel of any map in the \v{C}ech complex is holonomic by \Cref{PropSubQuotSum}\ref{item: holonomicity is preserved by submodules and quotients}, and since each $H^j_I(M)$ is a quotient of one of these kernels, it is holonomic, again by \Cref{PropSubQuotSum}\ref{item: holonomicity is preserved by submodules and quotients}.
\end{proof}

We now proceed to show that the previous result and \Cref{ThmLenHol} imply that the \v{C}ech cohomology modules of a holonomic $D$-module have finite sets of associated primes.
Toward that end, we first show that simple $D$-modules have at most one associated prime.

\begin{lemma}\label{lem: simple modules have few associated primes}
   Let $R$ be a commutative $\kk$-algebra. 
   If $M$ is a simple $D_{R|\kk}$-module, then $M$ has at most one associated prime.
\end{lemma}

\begin{proof}
   This follows from \cite[Lemmas~3.3.16 and 3.3.17]{Bjork79}, which we reproduce here for the reader's convenience.
   Given a simple left $D_{R|\kk}$-module $M$, let $u \in M\smallsetminus \{0\}$, so that $M = D_{R|\kk}u$.
   We claim that $\p \coloneqq \sqrt{\ann_R(u)}$ is prime and independent of the choice of $u$.
   From this, it follows easily that $\p$ is the only possible associated prime of $M$ (though $M$ could potentially fail to have associated primes).

   To verify that $\p$ does not depend on the choice of $u$, let $v$ be another nonzero element of $M$.
   Then there exists $\delta \in D_{R|\kk}$ such that $v = \delta u$.
   Suppose $f\in R$ is such that $f^i u = 0$ for some $i$.
   If $\delta$ has order $j$, then \Cref{commute-f} tells us that $f^{i+j} \delta \in D_{R|\kk} f^i$, and consequently, $f^{i+j}v =0$.
   This shows that $\sqrt{\ann_R(u)} \subseteq \sqrt{\ann_R(v)}$, and switching the roles of $u$ and $v$ we get the reverse containment.

   To verify that $\p$ is prime, let $f,g\in R$ and suppose that $fg \in \p$, but $g \notin \p$.
   Thus, $(fg)^iu=0$ for some $i$, but $v \coloneqq g^i u \ne 0$, and it follows that $f \in \sqrt{\ann_R(v)} = \p$.
   Lastly, note that $\p$ is proper, since $M$ is nontrivial.   
\end{proof}

\begin{proposition}
   \label{prop: finite ass}
   Let $R$ be a commutative $\kk$-algebra.
   Let $\filtration{F}{\bullet}{}$ be a filtration on $D_{R|\kk}$ that is linearly dominated by the order filtration.
   Suppose that $(D_{R|\kk},\filtration{F}{\bullet}{})$ is linearly simple with finite dimension and positive finite multiplicity.
   If $M$ is a holonomic $D_{R|\kk}$-module, then $\Ass_RH^i_I(M)$ is a finite set for every finitely generated ideal $I\subseteq R$ and $i \in \NN$.
\end{proposition}

\begin{proof}
   If $M$ is a holonomic $D_{R|\kk}$-module, then so is $H^i_I(M)$ by \Cref{ThmLocCohHol}(ii).
   \Cref{ThmLenHol} then shows that $H^i_I(M)$ has finite length as a $D_{R|\kk}$-module, so there exists an ascending chain 
   \[0=M_0\subseteq M_1\subseteq \cdots \subseteq M_{\ell}= H^i_I(M)\]
   of $D_{R|\kk}$-modules such that each quotient $M_{j}/M_{j-1}$ is simple.
   We conclude that $\Ass_RH^i_I(M)$ is contained in $\bigcup_{j=1}^{\ell} \Ass_R M_{j}/M_{j-1}$, a finite set by \Cref{lem: simple modules have few associated primes}.
\end{proof}

As a side note, we point out that a prime ideal that is minimal over the annihilator of an element of a module is called a \emph{weakly associated prime} of the module \cite[p.~341, Exercise~17]{bourbaki}.
In \Cref{lem: simple modules have few associated primes} we actually showed that simple $D$-modules have a unique weakly associated prime, while in \Cref{prop: finite ass} we showed that the \v{C}ech cohomology modules of a holonomic $D$-module have finitely many weakly associated primes.

\subsection{Holonomicity of the module $R_f[s] \fs$}

Let $R$ be a commutative algebra over a field $\kk$ and $f$ a nonzero element of $R$.
Another $D$-module of classical importance in the smooth case is the $D_{R|\kk}[s]$-module $R_f[s] \fs$, where $s$ is an indeterminate and  $\fs$ is a formal symbol.
This module is characterized by the property that for any integer $t$, any $\delta\in D_{R|\kk}[s]$, and any $a(s) \fs\in R_f[s] \fs$, the equality
 \[ (\delta \act a(s)\fs)\big|_{s\mapsto t} = \delta\big|_{s\mapsto t} (a(t) f^t)\]
 holds, where $(-)|_{s\mapsto t}$ denotes specialization of the indeterminate $s$ to the integer $t$.
 
Over a field $\kk$ of characteristic zero, any finitely generated or complete local $\kk$-algebra admits a unique such module up to isomorphism \cite[Theorem~2.12]{square}.
Note that this module is denoted by $M^R[\fs]$ in \loccit

We first give a more concrete description of this module. For $\delta\in D_{R|\kk}[s]$, we define the \emph{order} of $\delta$ to be $\min\{i\in \NN \ | \ \delta\in D^i_{R|\kk}[s]\}$.

\begin{lemma}\label{explicit-action}
   Let $R$ be a finitely generated commutative algebra over a field $\kk$ of characteristic zero, and $f$ a nonzero element of $R$.
   For $\delta\in D_{R|\kk}[s]$,  set $\delta^{(0)}=\delta$, and for each positive integer $i$  define $\delta^{(i)} \in D_{R|\kk}[s]$ inductively by $\delta^{(i)}=[\delta^{(i-1)},f]$. 
Then the action of $D_{R|\kk}[s]$ on $R_f[s] \fs$ is given by the rule
\[ \delta \act a(s) \fs = \sum_{i=0}^{\mathrm{ord}(\delta)} \binom{s}{i} f^{-i} \delta^{(i)} (a(s)) \fs,\]
where $\binom{s}{i}$ is the polynomial $\frac{s\cdot (s-1)\cdots(s-i+1)}{i!}$.	
\end{lemma}
\begin{proof}
First we show that this rule yields a module action. Let $\Theta: D_{R|\kk} \to D_{R_f|\kk}[s]$ be given by the rule
\[\Theta(\delta) = \sum_{i=0}^{\mathrm{ord}(\delta)} \binom{s}{i} f^{-i} \delta^{(i)}.\] We claim that this is a ring homomorphism. For this, it suffices to show that  $\Theta$ respects multiplication. To see this, first observe that by a straightforward induction, for any $\alpha,\beta\in D_{R|\kk}$ we have 
\[(\alpha \beta)^{(i)} = \sum_{j=0}^i \binom{i}{j} \alpha^{(j)} \beta^{(i-j)}.\]
Then
\begin{align}
  \Theta(\alpha \beta)
  &= \sum_{i=0}^{\mathrm{ord}(\alpha \beta)} \binom{s}{i} f^{-i} (\alpha \beta)^{(i)} \nonumber \\
  &= \sum_{i=0}^{\mathrm{ord}(\alpha \beta)} \binom{s}{i} f^{-i} \sum_{j=0}^i \binom{i}{j} \alpha^{(j)} \beta^{(i-j)}\nonumber \\
  &= \sum_{j,k\geq 0} \binom{s}{j+k} \binom{j+k}{j} f^{-j-k} \alpha^{(j)} \beta^{(k)}\nonumber \\
  &= \sum_{j,k\geq 0} \binom{s}{k}\binom{s-k}{j} f^{-j-k} \alpha^{(j)} \beta^{(k)}.\label{eq: theta of alpha beta}
   \end{align} 
   On the other hand,
     \begin{align*} \Theta(\alpha) \Theta(\beta) &= \left(\sum_{i=0}^{\mathrm{ord}(\alpha)} \binom{s}{i} f^{-i} \alpha^{(i)}\right)\left( \sum_{j=0}^{\mathrm{ord}(\beta)} \binom{s}{j} f^{-j} \beta^{(j)} \right)
       \\&= \sum_{i,j\ge 0}\binom{s}{i} \binom{s}{j} f^{-i} \left( \alpha^{(i)} f^{-j} \right) \beta^{(j)},
     \end{align*}
     and applying \Cref{commute-f} to $\alpha^{(i)} f^{-j}$ we see that
     \[ \Theta(\alpha) \Theta(\beta)  = \sum_{i,j,k\ge 0} \binom{s}{i} \binom{s}{j} \binom{-j}{k} f^{-i-j-k} \alpha^{(i+k)}  \beta^{(j)}.\]
  Grouping the summands according to the value of $l\coloneqq i+k$, this can be rewritten as 
  \[ \Theta(\alpha) \Theta(\beta)  = \sum_{j,l\ge 0} \left(\sum_{i,k\ge 0 \atop i+k=l} \binom{s}{i} \binom{-j}{k}\right) \binom{s}{j} f^{-l-j} \alpha^{(l)}  \beta^{(j)}\]
  where the sum in parentheses is a Vandermonde convolution that adds up to $\binom{s-j}{l}$, so this expression indeed equals \eqref{eq: theta of alpha beta}.
  Thus $\Theta$ is a homomorphism. We can extend $\Theta$ to $D_{R|\kk}[s] \to D_{R_f|\kk}[s]$ by setting $\Theta(s)=s$; we use the same name $\Theta$ for this extension.
  
Now note that there is a $D_{R_f|\kk}[s]$-action $\club$ on $R_f[s] \fs$ given by taking the standard action of $D_{R_f|\kk}$ on $R_f$ and extending in such a way that $s$ commutes with every operator. The action defined in the statement agrees with the action  $\delta \act a(s) \fs \coloneqq \Theta(\delta) \club a(s) \fs$. It follows that this is a module action.

Next, given $a(s) \fs \in R_f[s] \fs$, $\delta \in D_{R|\kk}[s]$, and $t\in \ZZ$, we have 
\begin{align*}
  (\delta \act a(s) \fs )\big|_{s\mapsto t} &= \left(\sum_{i=0}^{\mathrm{ord}(\delta)} \binom{s}{i} f^{-i} \delta^{(i)}(a(s))\fs \right)\Bigg|_{s\mapsto t} \\
                                  &= \sum_{i=0}^{\mathrm{ord}(\delta)} \binom{t}{i} f^{t-i} \big(\delta\big|_{s\mapsto t}\big)^{(i)}(a(t)) \\
                                  &= (\delta\big|_{s\mapsto t} f^t) (a(t)) \\
                                  &= \delta\big|_{s\mapsto t} ( a(t) f^t).
   \end{align*}
We conclude that the module with the stated action is isomorphic to $R_f[s] \fs$ \cite[Theorem~2.12]{square}. 
\end{proof}

We will also use the following, which is well known to the experts (see for instance \cite[Proposition~II.3.13]{KunzADC}).

\begin{lemma}\label{lem:basechange} Let $\kk$ be a field, $R$ a finitely generated commutative $\kk$-algebra, and $s$ an indeterminate. Write $R(s)$ for $R\otimes_{\kk} \kk(s)$. For every $i\in\NN$, there is an isomorphism $D^i_{R(s) | \kk(s) } \cong D^i_{R|\kk} \otimes_{\kk} \kk(s)$.
\qed
\end{lemma}

\begin{theorem}\label{thm: R_f(s) f^s is holonomic}
 Let $R$ be a commutative algebra over a field $\kk$ of characteristic zero, and $f$ a nonzero element of $R$.
   Let $\filtration{F}{\bullet}{}$ be a filtration on $D_{R|\kk}$ that is linearly dominated by the order filtration.
   Suppose that $(D_{R|\kk},\filtration{F}{\bullet}{})$ is linearly simple with finite dimension and finite positive multiplicity. Set $\tildefiltration{F}{i}{}=\kk(s) \otimes_{\kk} \filtration{F}{i}{}$ for all $i$.
   If $R$ is holonomic as a $D_{R|\kk}$-module, then $R_f(s) \fs$ is a holonomic $(D_{R(s)|\kk(s)},\tildefiltration{F}{\bullet}{})$-module.
\end{theorem}
\begin{proof}
First, by \Cref{lem:basechange}, we see that $\tildefiltration{F}{\bullet}{}$ is a filtration of $D_{R(s)|\kk(s)}$. As $\dim_{\kk} \filtration{F}{i}{} = \dim_{\kk(s)} \tildefiltration{F}{i}{}$, we have \[\Dim (D_{R(s)|\kk(s)},\tildefiltration{F}{\bullet}{}) = \Dim (D_{R|\kk},\filtration{F}{\bullet}{}) \ \ \text{and} \ \ \e (D_{R(s)|\kk(s)},\tildefiltration{F}{\bullet}{}) = \e (D_{R|\kk},\filtration{F}{\bullet}{}).\] The linear simplicity of $(D_{R(s)|\kk(s)},\tildefiltration{F}{\bullet}{})$ also follows from that of $(D_{R|\kk},\filtration{F}{\bullet}{})$, with the same constant.

Choose $a$ such that $f\in \filtration{F}{a}{}$; fix $C$ such that $\filtration{F}{i}{} \subseteq D_{R|\kk}^{Ci}$ for all~$i$.
Let $\filtration{G}{\bullet}{}$ be a filtration on $R$
compatible with $\filtration{F}{\bullet}{}$
of dimension $\frac{1}{2} \Dim(\filtration{F}{\bullet}{})$ and finite multiplicity.
Consider the filtration $\tildefiltration{G}{\bu}{}$ on $R_f(s)$ given by
 \[\tildefiltration{G}{j}{} = \kk(s) \otimes_{\kk} \frac{1}{f^{Cj}} \filtration{G}{j(aC+1)}{}.\]
    Let $\delta\in \tildefiltration{F}{i}{}$, so that $\ord(\delta) \le Ci$.
    Using \Cref{commute-f,explicit-action}, we see that
 \begin{align*} \delta \act \big(\tildefiltration{G}{j}{}\fs \big) 
  &\subseteq \left(\sum_{k=0}^{Ci} f^{-k} \delta^{(k)}  \tildefiltration{G}{j}{} \right) \fs \\  
 &\subseteq \left(\sum_{k=0}^{Ci} \sum_{l=0}^{Ci-k} f^{-k-l-Cj} \delta^{(k+l)} f^{Cj} \tildefiltration{G}{j}{}\right) \fs \\
 &\subseteq  \left(\sum_{k=0}^{Ci}  f^{-C(i+j)} \delta^{(k)} f^{Cj} \tildefiltration{G}{j}{} \right) \fs.
 \end{align*}
 Now observe, as in the proof of \Cref{LemmaLocSameDim}, that $\delta^{(k)}\in \tildefiltration{F}{i+ak}{}$ for every $k$, so each $\delta^{(k)}$ in the last sum lies in $\tildefiltration{F}{i(aC+1)}{} = \kk(s)\otimes_\kk \filtration{F}{i(aC+1)}{}$.
 As $f^{Cj} \tildefiltration{G}{j}{} = \kk(s)\otimes_\kk \filtration{G}{j(aC+1)}{}$, the compatibility of $\filtration{G}{\bullet}{}$ with $\filtration{F}{\bullet}{}$ allows us to conclude that 
 \[\delta \act \big(\tildefiltration{G}{j}{} \fs\big) \subseteq  \kk(s)\otimes_\kk \frac1{f^{C(i+j)}}\filtration{G}{(i+j)(aC+1)}{} \fs = \tildefiltration{G}{i+j}{} \fs. \]
 Thus, $\tildefiltration{G}{\bullet}{}\fs$ is a filtration of $R_f(s) \fs$  compatible with $\tildefiltration{F}{\bullet}{}$.
 As in \Cref{ThmLenHol}, one can compute the lengths and see that $R_f(s) \fs$ is holonomic.
\end{proof}

   \begin{corollary}\label{CorBS1}
      Let $R$ be a finitely generated commutative algebra over a field $\kk$ of characteristic zero.
   Let $\filtration{F}{\bullet}{}$ be a filtration on $D_{R|\kk}$ that is linearly dominated by the order filtration.
   Suppose that $(D_{R|\kk},\filtration{F}{\bullet}{})$ is linearly simple with finite dimension and finite positive multiplicity. If $R$ is holonomic as a $D_{R|\kk}$-module, then for every nonzero element $f\in R$, $R_f(s) \fs$ has finite length as a $D_{R(s)|\kk(s)}$-module. \qed 
\end{corollary}

As in the classical case of smooth $\kk$-algebras \cite{Ber72,SatoPoly}, we can also deduce the existence of Bernstein--Sato polynomials in this singular setting. Namely, 
 for every $f\in R$, there exist $\delta\in D_{R|\KK}[s]$ and a nonzero polynomial $b(s)\in \KK[s]$ such that the following functional equation is satisfied:
\begin{equation*}\label{eq: BS functional eq for single poly}
\delta \act f\fs=b(s) \fs.
\end{equation*}
The \emph{Bernstein--Sato polynomial} associated to $f$ is the unique monic polynomial  of smallest degree satisfying such functional equation.
See \cite{square} for more insight on the Bernstein--Sato theory in singular $\kk$-algebras, and the recent survey \cite{SurveyBS} for more information on the role of this polynomial in algebraic geometry and commutative algebra.

\begin{corollary}\label{CorBS2}
   Let $R$ be a finitely generated commutative algebra over a field $\kk$ of characteristic zero.
   Let $\filtration{F}{\bullet}{}$ be a filtration on $D_{R|\kk}$ that is linearly dominated by the order filtration.
   Suppose that $(D_{R|\kk},\filtration{F}{\bullet}{})$ is linearly simple with finite dimension and finite positive multiplicity. If $R$ is holonomic as a $D_{R|\kk}$-module, then every nonzero element $f\in R$  admits a Bernstein--Sato polynomial. \qed
\end{corollary}

\subsection{Generalized Bernstein filtrations}\label{sec: gbf}

The aim of this subsection is to extend the Bernstein filtration beyond the polynomial ring case. 
We assume the following:

\begin{setup} \label{Notation_graded}
   Let $\kk$ be a field,
   and $R$ a \cga, as in \Cref{conv: commutative graded ring}.
	Set $\m=\bigoplus_{i>0} R_i$, $n=\dim_\kk \m/\m^2$, and $w=\max\{t\in\NN\;|\;[ \m/\m^2]_t\neq 0\}$.
	We fix a polynomial ring $S=\kk[x_1,\ldots,x_n]$ over $\kk$, a grading on $S$ with $\deg x_i=w_i\in \PP$, and a homogeneous ideal $I\subseteq S$ such that $R\cong S/I$ as graded rings.
        We observe that $w=\max\{w_1,\ldots,w_n\}$.
\end{setup}

\begin{definition} \label{Bernstein_filtration}
   Let $R$ be as in \Cref{Notation_graded}, and $a$ a real number greater than $w$.
   The \emph{generalized Bernstein filtration $\gbf{\bullet}{a,R}$ on $D_{R|\kk}$ with slope $a$} is given by
   \[ \gbf{i}{a,R}\coloneqq \kk \cdot \{ \delta \in D_{R|\kk} \ \text{homogeneous} \ | \ \deg(\delta) + a \ord(\delta) \leq i \}.\]
   If the slope is clear from the context, or irrelevant, then we simply write $\gbf{\bullet}{R}$.
\end{definition}

\begin{example}
   \label{ex: Bernstein filtration of positively graded polynomial rings}
   Let $S$ be a standard graded polynomial ring over a field of characteristic zero. Then the generalized Bernstein filtration with slope $2$ is just the usual Bernstein filtration on the Weyl algebra.
   If $S=\kk[x_1,\dots,x_n]$ is a positively graded polynomial ring with $\deg(x_i)=w_i$, we may interpret the generalized Bernstein filtration with integral slope $a$, in the context of \Cref{rem: weights}, as the filtration associated to the weight vector $(\underline{w}, a-\underline{w})\coloneqq (w_1,\dots,w_n, a-w_1,\dots, a-w_n)\in \ZZ^{2n}$.			
\end{example}

\begin{lemma}\label{RemAdmSlopeW}
   If $R$ is as in \Cref{Notation_graded} and $\gbf{\bullet}{a,R}$ is a generalized Bernstein filtration on $D_{R|\kk}$, then the following hold.
\begin{enumerate}
\item 	$\dim_\kk \gbf{i}{a,R}$ is finite for every $i$.
\item There exists $\varepsilon>0$ such that $\deg(\delta) + (a-\varepsilon) \ord(\delta) > 0$ for all homogeneous $\delta\in D_{R|\kk}\smallsetminus \kk$.
\end{enumerate}	
\end{lemma}
\begin{proof}
	By explicit computation, this holds for a polynomial ring. It then follows for $R$ from the fact that a differential operator on a quotient of a polynomial ring is the image of a differential operator on the polynomial ring.
\end{proof}

\begin{proposition}
   If $R$ is as in \Cref{Notation_graded}, then every generalized Bernstein filtration on $D_{R|\kk}$ is a $\kk$-algebra filtration in the sense of \Cref{notation:filtration}.
\end{proposition}

\begin{proof}
   The generalized Bernstein filtration $\gbf{\bullet}{a,R}$ satisfies $\gbf{i}{a,R}\gbf{j}{a,R} \subseteq \gbf{i+j}{a,R}$ for each $i$ and $j$ because $\ord(\delta_1\delta_2)\leq \ord(\delta_1)+\ord(\delta_2)$ and $\deg(\delta_1\delta_2)\leq \deg(\delta_1)+\deg(\delta_2)$ for any homogeneous operators $\delta_1,\delta_2 \in D_{R|\kk}$.
   This filtration is  ascending, and it is finite dimensional by \Cref{RemAdmSlopeW}(i).
   It is exhaustive because any homogeneous operator $\delta \in D_{R|\kk}$ lies in $\gbf{i}{a,R}$ for $i=\lceil\deg(\delta)+a\ord(\delta)\rceil$.
   Finally, if $\delta$ is a homogeneous element of $\gbf{0}{a,R}$, then $\deg(\delta) + a \ord(\delta) \le 0$, so \Cref{RemAdmSlopeW}(ii) tells us that $\delta\in \kk$.
   Thus, $\gbf{0}{a,R} \subseteq \kk$, and since the reverse inclusion is obvious, equality holds.
\end{proof} 

\begin{proposition}\label{prop:compare-gbf-order}
   If $R$ is as in \Cref{Notation_graded} and $\gbf{\bullet}{R}$ is a generalized Bernstein filtration on $D_{R|\kk}$, then $\gbf{\bullet}{R}$ is linearly dominated by the order filtration.
\end{proposition}

\begin{proof}
   Suppose $\gbf{\bullet}{R}$ has slope $a$.
   By \Cref{RemAdmSlopeW}(ii) we have some $\varepsilon>0$ such that $\deg(\delta) + (a-\varepsilon) \ord(\delta) >0$ for every homogeneous operator $\delta$ not in $\kk$.
   Set $C = \lceil 1/\varepsilon\rceil$.
   Then, for $\delta$ homogeneous in $\gbf{i}{R} \smallsetminus \kk$ we have $\varepsilon \ord(\delta) < \deg(\delta) + a \ord(\delta) \leq i$, so $\ord(\delta) < Ci$, from which it follows that $\gbf{i}{R} \subseteq D_{R|\kk}^{Ci}$.
\end{proof}

Next, we show that all generalized Bernstein filtrations on $D_{R|\kk}$ are linearly equivalent.

\begin{proposition}\label{prop:gbf-lin-eq}
   If $R$ is as in \Cref{Notation_graded}, then any two generalized Bernstein filtrations $\gbf{\bullet}{a,R}$ and $\gbf{\bullet}{b,R}$ on $D_{R|\kk}$ are linearly equivalent.
\end{proposition}

\begin{proof}
   If $a\leq b$, choose an integer $C$ such that $b\leq C (a - w)$.
   Then we have $\gbf{i}{b,R}\subseteq \gbf{i}{a,R}$ and $\gbf{i}{a,R}\subseteq \gbf{Ci}{b,R}$ for all $i$.
\end{proof}

The following is an immediate consequence of \Cref{rem: linear simplicity only depends on linear equivalence class,prop:gbf-lin-eq}.

\begin{proposition}
   \label{prop: linear simplicity is independent of slope}
   If $R$ is as in \Cref{Notation_graded}, and $\gbf{\bullet}{a,R}$ and $\gbf{\bullet}{b,R}$ are generalized Bernstein filtrations on $D_{R|\kk}$, then $(D_{R|\kk},\gbf{\bullet}{a,R})$ is linearly simple if and only if $(D_{R|\kk},\gbf{\bullet}{b,R})$ is linearly simple.
   \qed
\end{proposition}

Other consequences of \Cref{prop:gbf-lin-eq}---that the dimension, and the positivity and finiteness of the multiplicity of a Bernstein filtration are all independent of the slope---are explored in the next subsection.
Coming back to the case of a polynomial ring $S=\kk[x_1,\dots,x_n]$ over a field $\kk$ of characteristic zero, we now show that $D_{S|\kk}$ with a generalized Bernstein filtration is linearly simple.

\begin{proposition}\label{prop:bavula-poly}
   Let $S$ be a positively graded polynomial ring over a field $\kk$ of characteristic zero.
   If $\gbf{\bullet}{S}$ is a generalized Bernstein filtration on $D_{S|\kk}$, then $(D_{S|\kk},\gbf{\bullet}{S})$ is linearly simple; that is, there exists a positive integer $C$ such that for every $i\in \NN$ and every nonzero $\delta\in \gbf{i}{S}$, we have $1\in \gbf{Ci}{S} \cdot \delta \cdot \gbf{Ci}{S}$.
\end{proposition}

\begin{proof}
   Suppose $\gbf{\bullet}{S}$ has slope $a$, which we may assume is an integer, by \Cref{prop: linear simplicity is independent of slope}.
   Suppose that $S = \kk[x_1,\ldots, x_n]$ and $\deg(x_j) = -\deg(\partial_j) = w_j$ for each $j$, so that $x_j \in \gbf{w_j}{S}$  and $\partial_j \in \gbf{a-w_j}{S}$.
   We shall verify that the claim holds for $C = \max\{w_j,a-w_j\}$.
   This is true when $i=0$, so suppose that $i>0$ and $\delta\in \gbf{i}{S}\smallsetminus\{0\}$.
   Let $\varepsilon_j$ denote the $j$-th standard basis vector of $\ZZ^n$.
   The identity $[ x^{\alpha}\partial^{\beta}, \partial_j] = -\alpha_j x^{\alpha-\varepsilon_j}\partial^{\beta}$ shows that, as long as the normal form of $\delta$ contains a monomial $x^\alpha \partial^\beta$ with $\alpha_j \ne 0$ for some $j$, the commutator $[\delta,\partial_j]$ is a nonzero element of $\gbf{i-w_j}{S}$.
   An inductive argument then allows us to assume that $1 \in \gbf{C(i-w_j)}{S} [\delta,\partial_j]\gbf{C(i-w_j)}{S}$.
   But $[\delta,\partial_j]$ lies in $\gbf{a-w_j}{S}\cdot\delta\cdot\gbf{a-w_j}{S}$, so we conclude that
   \[1\in \gbf{C(i-w_j)}{S}\gbf{a-w_j}{S}\cdot\delta\cdot\gbf{a-w_j}{S}\gbf{C(i-w_j)}{S}\subseteq \gbf{Ci}{S} \cdot \delta \cdot \gbf{Ci}{S}\]
   where the last containment follows from the inequalities $C \ge a-w_j$ and $w_j \ge 1$.

   If the normal form of $\delta$ contains a monomial $\partial^\beta$ with $\beta_j \ne 0$ for some $j$, we mimic the previous argument using, instead, the identity $[\partial^{\beta}, x_j] = \beta_j \partial^{\beta-\varepsilon_j}$.
   This time around, we use the inequalities $C\ge w_j$ and $a-w_j \ge 1$, the latter being a consequence of our assumption that $a\in \NN$ and $a>w$.

   Finally, if neither condition is satisfied, then $\delta \in \kk\smallsetminus \{0\}$, and the result follows.
\end{proof}

We conclude this subsection by showing that the associated graded ring of a generalized Bernstein filtration on a polynomial ring is itself a polynomial ring.

\begin{proposition}
   \label{prop: associated graded ring is polynomial ring}
   Let $S=\kk[x_1,\dots,x_n]$ be a positively graded polynomial ring over a field $\kk$ of characteristic zero.
   If $\gbf{\bullet}{S}$ is a generalized Bernstein filtration on $D_{S|\kk}$ with an integral slope, then the associated graded ring $\gr(\gbf{\bullet}{S})$ is a polynomial ring in $2n$ variables, and in particular, is commutative.
\end{proposition}
 
\begin{proof}
   Let $a \in \NN$ be the slope of $\gbf{\bullet}{S}$, and assume that $\deg(x_j)= - \deg(\partial_j) = w_j$, for each~$j$.
   As in \Cref{ex: Bernstein filtration of positively graded polynomial rings}, $\gbf{\bullet}{S}$ is the filtration associated to the weight vector $(\underline{w}, a-\underline{w}) \in \ZZ^{2n}$.
   Since $w_j + (a-w_j) > 0$ for all~$j$, the associated graded ring $\gr(\gbf{\bullet}{S})$ is a polynomial ring in $2n$ variables \cite[Proposition~2.2, Example~2.4]{GregSmith}.
\end{proof}

\subsection{Bernstein algebras}

We retain the hypotheses of \Cref{Notation_graded}.
In order to have a good theory of holonomic $D_{R|\kk}$-modules generalizing what is known in the regular case, we would like not only to have Bernstein's inequality, but also to have the dimension of $D_{R|\kk}$ be twice the dimension of $R$.
We note first that dimension for the ring of differential operators is not well defined in general.

\begin{example}
   Let $\kk=\FF_p$ and $R=\kk[x_1,\dots,x_d]$.
   Note that $D_{R|\kk}$ is not finitely generated, so the notion of standard filtration for $D_{R|\kk}$ as a $\kk$-algebra as in \Cref{def:algebra-std-filtration} does not apply.
   If we set $\filtration{F}{i}{}=\big[D_{R|\kk}^{(\ell_i)}\big]_{\leq i}$ where $\ell_i=\lfloor \log_p(\log_p(i))\rfloor$, then one verifies easily that $\filtration{F}{\bullet}{}$ is a filtration in the sense of \Cref{notation:filtration} and that $\Dim(\filtration{F}{\bullet}{})=d$.
   \Cref{prop:any-larger-dim} then tells us that for any $\lambda\geq d$, there exists a filtration on $D_{R|\kk}$ with dimension $\lambda$.
\end{example}

However, the generalized Bernstein filtrations yield a well-defined notion of dimension for the ring of differential operators, and finitely generated $D_{R|\kk}$-modules, as the following result shows.

\begin{proposition}
   \label{prop: independence of slope}
   If $R$ is a \cga, as in \Cref{Notation_graded}, and $\gbf{\bullet}{a,R}$ and $\gbf{\bullet}{b,R}$ are generalized Bernstein filtrations on $D_{R|\kk}$, then 
   \begin{enumerate}
      \item $\Dim(\gbf{\bullet}{a,R})=\Dim(\gbf{\bullet}{b,R})$,
      \item $\e(\gbf{\bullet}{a,R})<\infty$ if and only if $\e(\gbf{\bullet}{b,R})<\infty$, and
      \item $\e(\gbf{\bullet}{a,R})>0$ if and only if $\e(\gbf{\bullet}{b,R})>0$.
   \end{enumerate}
   Moreover, if $M$ is a finitely generated left or right $D_{R|\kk}$-module, then, in the notation of \Cref{def:module-mult-alg-filt},
   \begin{enumerate}\setcounter{enumi}{3}
      \item $\Dim(M,\gbf{\bullet}{a,R})=\Dim(M,\gbf{\bullet}{b,R})$,
      \item $\e(M,\gbf{\bullet}{a,R})<\infty$ if and only if $\e(M,\gbf{\bullet}{b,R})<\infty$, and
      \item $\e(M,\gbf{\bullet}{a,R})>0$ if and only if $\e(M,\gbf{\bullet}{b,R})>0$.
   \end{enumerate}
\end{proposition}

\begin{proof}
   This follows from \Cref{prop:lin-equiv,prop:lin-equiv-alg-mod,prop:gbf-lin-eq}.
\end{proof}

\begin{definition}
   Let $R$ be as in \Cref{Notation_graded}, and $M$ a finitely generated $D_{R|\kk}$-module.
   Then $\dim(D_{R|\kk}) \coloneqq \Dim(\gbf{\bullet}{R})$ and $\dim(M) \coloneqq \Dim(M,\gbf{\bullet}{R})$, where $\gbf{\bullet}{R}$ is a generalized Bernstein filtration on $D_{R|\kk}$.
\end{definition}

\Cref{prop: independence of slope} ensures that these definitions do not depend on the choice of the generalized Bernstein filtration. 

The class of algebras for which we have a good theory of holonomic $D_{R|\kk}$-modules is the following.

\begin{definition} \label{Bernstein_algebra}
   Let $R$ be as in \Cref{Notation_graded}, and $\gbf{\bullet}{R}$ a generalized Bernstein filtration on $D_{R|\kk}$.
   We say that $R$ is a \emph{Bernstein algebra} if 
   \begin{enumerate}
      \item $(D_{R|\kk},\gbf{\bullet}{R})$ is linearly simple,
      \item 	$\Dim(\gbf{\bullet}{R})=2\dim(R)$, and
      \item $0<\e(\gbf{\bullet}{R})<\infty$.
   \end{enumerate}
\end{definition}

Note that \Cref{prop: independence of slope,prop: linear simplicity is independent of slope} show that conditions (i)--(iii) in the previous definition do not depend on the choice of the slope for $\gbf{\bullet}{R}$.

We give sufficient conditions for a \cga\ to be a Bernstein algebra in terms of the \emph{differential signature} introduced by Brenner and the third and fourth authors of this manuscript \cite{BJNB}.

\begin{definition}[\cite{SurveySP,BJNB}]
   Let $R$ be a \cga\ of dimension $d$, with maximal homogeneous ideal $\m$.
   \begin{enumerate}
      \item For each positive integer $i$, the \emph{$i$-th differential power of $\m$} is the ideal
      \[\m\dif{i}=\big\{f\in R \, | \, \delta(f)\in \m \hbox{ for all } \delta\in D^{i-1}_{R|\kk}\big\}.\]
      \item The \emph{differential signature of $R$} is the real number
      \[\ds(R)=\limsup\limits_{i\to\infty} \frac{d!\cdot\dim_\kk R/\m\dif{i}}{i^d}.\]
   \end{enumerate}
\end{definition}

\begin{remark} \label{positive-diff-sig: R}
   Let $G$ be a finite group that acts linearly on a polynomial ring $R$ over a field $\kk$ of characteristic zero.
   Then the differential signature of the ring of invariants $R^G$ is positive \cite[Theorem~6.15]{BJNB}.
   Moreover, the differential signature of a strongly $F$-regular $\kk$-algebra is also positive \cite[Theorem~5.17]{BJNB}.
\end{remark}

The next result presents a perfect pairing between a certain quotient of $D^{i-1}_{R|\kk}$ and $R/\m\dif{i}$.
This was implicitly introduced in previous work regarding convergence of differential signature \cite[Section~8]{BJNB}.

\begin{lemma}[{\cite[Lemma~3.4]{DNB}}]\label{LemmaPairing}
   If $R$ is a \cga\ with homogeneous maximal ideal $\m$, and
   \[\cJ_{R|\kk}=\{\delta \in D_{R|\kk} \;|\; \delta(R)\subseteq \m\},\]
   then there exists a non-degenerate $\kk$-bilinear function 
\[D^{i-1}_{R|\kk}/\big(\cJ_{R|\kk} \cap D^{i-1}_{R|\kk}\big)\times R/\m\dif{i} \to R/\m\]
defined by $( \overline{\delta },\overline{r} ) \mapsto \overline{\delta(r)}$.
\qed
\end{lemma}

\begin{lemma}\label{LemmaFiltR}
   Let $R$ be as in \Cref{Notation_graded}, and $\gbf{\bullet}{R}$ a generalized Bernstein filtration on $D_{R|\kk}$.
   If $\filtration{G}{\bullet}{}$ is the $D_{R|\kk}$-module filtration on $R$ given by $\filtration{G}{i}{}=\gbf{i}{R} \cdot 1 \subseteq R$, then $\Dim(\filtration{G}{\bullet}{})=\dim(R)$, and $\e(\filtration{G}{\bullet}{}) = \e(R)$, which is finite and positive.
\end{lemma}

\begin{proof}
   As $\filtration{G}{i}{}=\gbf{i}{R} \cdot 1 = [R]_{\leq i}$, the result follows from the theory of Hilbert functions for commutative finitely generated graded $\kk$-algebras. 
\end{proof}

\begin{theorem}\label{thm:dim-Bernstein-filtration}
   Let $R$ be as in \Cref{Notation_graded}, and $\gbf{\bullet}{R}$ a generalized Bernstein filtration on $D_{R|\kk}$.
   If $\ds(R)>0$ and $(D_{R|\kk},\gbf{\bullet}{R})$ is linearly simple, then $R$ is a Bernstein algebra.
   In particular, $R$, its localizations $R_f$ for $f\in R$, and its local cohomology modules $H^j_I(R)$ for $I\subseteq R$ are all holonomic $D_{R|\kk}$-modules.
  Likewise, if $\kk$ has characteristic zero, then $R_f(s) \fs$ is a holonomic $D_{R(s)|\kk(s)}$-module for any nonzero element $f\in R$.
\end{theorem}

\begin{proof}
   For each positive integer $i$, set $\alpha_i=\dim_\kk R/\m\dif{i}$.
   By \Cref{LemmaPairing}, there exist $\delta_1,\ldots, \delta_{\alpha_i}\in D^{i-1}_{R|\kk}$ and $f_1,\ldots, f_{\alpha_i}\in R$ homogeneous such that $\delta_j(f_k)=0$ if $j\neq k$, and $\delta_j(f_j)=1$.
   In particular, this implies that $\delta_1,\ldots, \delta_{\alpha_i}$ are linearly independent over $R$.
   Let $a$ be the slope of $\gbf{\bullet}{R}$, and $d=\dim(R)$.
   Noting that $\deg(\delta_j)\leq 0$ for each $j$, we see that $[R]_{\leq i} \{  \delta_1,\ldots, \delta_{\alpha_i} \}\subseteq \gbf{(a+1)i}{R}$, and thus
   \[
      \limsup\limits_{i\to \infty}\frac{\dim_\kk \gbf{i}{R}}{i^{2d}} \geq
      \limsup\limits_{i\to \infty} \frac{\dim_\kk \gbf{(a+1)i}{R}}{((a+1)i)^{2d}} \geq
      \limsup\limits_{i\to \infty}\frac{\dim_\kk ( [R]_{\leq i} \{  \delta_1,\ldots, \delta_{\alpha_i}\} )}{((a+1)i)^{2d}}.
   \]
   The linear independence of the $\delta_j$ over $R$ implies that $\dim_\kk ( [R]_{\leq i} \{  \delta_1,\ldots, \delta_{\alpha_i}\} ) = \alpha_i \dim_\kk [R]_{\leq i}$, so
   \begin{align*}
     \limsup\limits_{i\to \infty}\frac{\dim_\kk \gbf{i}{R}}{i^{2d}}
     &\ge \limsup\limits_{i\to \infty}\frac{\alpha_i \dim_\kk [R]_{\leq i} }{((a+1)i)^{2d}}\\
     & = \frac{1}{(a+1)^{2d}}\cdot \limsup\limits_{i\to \infty}\frac{\alpha_i }{i^{d}} \cdot \lim\limits_{i\to \infty}\frac{ \dim_\kk [R]_{\leq i} }{i^{d}}\\
     &=\frac{1}{(a+1)^{2d}(d!)^2} \cdot\ds(R)\e(R).
   \end{align*}
   Since the last quantity is positive, this shows that $\Dim(\gbf{\bullet}{R})\geq 2d$.

   To prove the reverse inequality, let $\filtration{G}{\bullet}{}$ be as in \Cref{LemmaFiltR}.
   Then $\filtration{G}{\bullet}{}$ is a filtration on the $D_{R|\kk}$-module $R$ compatible with $\gbf{\bullet}{R}$, of dimension $d=\dim(R)$, so Bernstein's inequality (\Cref{thm:Bern}) tells us that $\Dim(\gbf{\bullet}{R})\leq 2d$.

   Having established that $\Dim(\gbf{\bullet}{R}) = 2d$, the calculation displayed above also shows that $\e(\gbf{\bullet}{R})$ is positive, whereas \Cref{thm:Bern} tells us that if $(D_{R|\kk},\gbf{\bullet}{R})$ is $C$-linearly simple, then
   \begin{align*}
\e (\gbf{\bullet}{R}) & \leq (C+1)^{d}(C+2)^{d}  \e (\filtration{G}{\bullet}{})^2\\
&= (C+1)^{d}(C+2)^{d}  \e (R)^2 <\infty.
   \end{align*}
   The claims about holonomicity follow from \Cref{ThmLocCohHol,prop:compare-gbf-order,LemmaFiltR,thm: R_f(s) f^s is holonomic}.
\end{proof}

In the final two sections of this paper, we shall use \Cref{thm:dim-Bernstein-filtration} to introduce two interesting classes of Bernstein algebras.

\section{Rings of invariants of finite groups in characteristic zero}  \label{Sec5}

Let $R$ be a polynomial ring over a field $\kk$ of characteristic zero, and $G$ a finite group acting linearly on $R$. 
The goal of this section is to prove that the ring of invariants $R^G$ is a Bernstein algebra.
In particular, Bernstein's inequality is satisfied, and the dimension of $D_{R^G|\kk}$ is twice the dimension of $R^G$.
We also relate holonomic modules over $R$ and over $R^G$.

\subsection{Bernstein inequality for rings of invariants}  \label{Sec5_1}

If $R$ and $G$ are as in the preceding paragraph, the action of $G$ on $R$ induces a degree and order-preserving action on $D_{R|\kk}$, defined as follows: for each $g\in G$ and $\delta\in D_{R|\kk}$, we define $g\cdot \delta \in D_{R|\kk}$ by
\[
   (g\cdot\delta)(r) \coloneqq g\cdot \delta(g^{-1}\cdot r).
\]
It is easy to verify that if $\delta \in (D_{R|\kk})^G$, then $\delta$ maps $R^G$ into itself, so we have a well-defined map
\begin{equation}
   \label{eq: restriction}
   (D_{R|\kk})^G \to D_{R^G|\kk}
\end{equation}
given by restriction.
Note that as $G$ is finite, this map is injective \cite[Theorem~6.3(1)]{Schwarz} (see also \cite[Theorem~2]{Traves} for an elementary proof).

For our applications, we need the restriction map \eqref{eq: restriction} to be not merely injective, but surjective as well, to allow us to relate generalized Bernstein filtrations on $D_{R|\kk}$ and on $D_{R^G|\kk}$.
This is made possible by the following result.

\begin{lemma}\label{lem:nopseudo}
   Let $G$ be a finite group that acts linearly on a polynomial ring $R$ over a field $\kk$.
   Assume that $|G|$ is nonzero in $\kk$.
   Then there exists a normal subgroup $H\trianglelefteq G$ such that $R^H$ is a polynomial ring \textup(that may not be standard graded\textup), and such that $ \big(D^i_{R^H|\kk}\big)^{G/H}\cong D^i_{R^G|\kk}$ under the natural restriction map for every $i$.
\end{lemma}

\begin{proof}
   We identify $R$ with the ring of polynomial functions on a $\kk$-vector space $V$.
   Let $H\leq G$ be the subgroup generated by all elements $g$ in $G$ such that the rank of $\id- g$ on $V$ is one; such elements are called pseudoreflections in the literature.
   The subgroup $H$ is normal, and $R^H$ is a polynomial ring by the Shephard--Todd Theorem \cite{ShephardTodd} (see also \cite[Theorem~7.1.4]{neusel} for a modern proof, in arbitrary characteristic).

   The inclusion map ${R^G \to R^H}$ is \'etale in codimension one.
   This is well known to experts, but we include an argument for convenience of the reader.
   Since the \'etale locus is preserved by base change, we may assume that $\kk$ is algebraically closed.
   Let $g_1,\dots,g_t \in G$ be a set of representatives for $G/H \smallsetminus \{H\}$. Let $X\subseteq V$ be the union of the fixed spaces of $g_1,\dots,g_t$.
   Note that $X$ has codimension at least two in $V$.
   It suffices to show that for any maximal ideal $\m$ of $R$ that corresponds to a point $v\in V\smallsetminus X$, the inclusion map ${R^G_{\m \cap R^G} \to R^H_{\m \cap R^H}}$ is \'etale.
   To this end, we note that the stabilizer of $v$ in $G$ is contained in $H$.
   Then, by~\cite[Proposition~1.1]{Kemper}, the inclusion map induces an isomorphism on the completions ${\widehat{R^G}_{\m \cap R^G} \xrightarrow{\cong} \widehat{R^H}_{\m \cap R^H}}$, and the claim follows.

   Thus,  the restriction map
   \[{D^i_{R^H|\kk} \to D^i_{R^G|\kk}\big(R^G,R^H\big)},\]
   where $D^i_{R^G|\kk}\big(R^G,R^H\big)$ denotes the differential operators from $R^G$ to $R^H$ as {$R^G$-modules}, is an isomorphism for all $i$ \cite[Proof of Proposition~6.4]{BJNB}.
   Restricting to $G/H$ invariants on both sides yields the desired isomorphism.
\end{proof}

The case of the previous lemma where $G$ contains no pseudoreflections is a theorem of Kantor \cite[Chapitre III, Th\'eor\`eme 4]{Kantor}.

With the notation of \Cref{lem:nopseudo}, we have $R^G = \big(R^H\big)^{G/H}$, where $R^H$ is a polynomial ring and the action of $G/H$ on $R^H$, although not necessarily linear, fixes $\kk$ and is degree preserving.
\Cref{lem:nopseudo}  allows us to work in the following setting.

\begin{setup}\label{setup:finite-char0}
   Let $R=\kk[x_1,\dots,x_n]$ be a polynomial ring over a field $\kk$ of characteristic zero, with an arbitrary positive grading.
   Let $G$ be a finite group that acts on $R$ by degree-preserving $\kk$-linear automorphisms, and suppose that $D_{R^G|\kk}\cong (D_{R|\kk})^G$ under the natural restriction map \eqref{eq: restriction}.
   Henceforth we freely identify $D_{R^G|\kk}$ with $(D_{R|\kk})^G$.
   Let $\gbf{\bullet}{R}$ and $\gbf{\bullet}{R^G}$ be generalized Bernstein filtrations on $D_{R|\kk}$ and $D_{R^G|\kk}$ with the \emph{same integral slope}.
   Then under the aforementioned identification we have
   \begin{equation}\label{eq: relating filtrations}
      \gbf{i}{R^G}=(\gbf{i}{R})^G = \gbf{i}{R}\cap D_{R^G|\kk}\ \text{ for each $i$.}
   \end{equation}
\end{setup}

\begin{lemma}\label{lem:invariants-fg-graded}
   In the context of \Cref{setup:finite-char0}, there exists a natural injective ring homomorphism $\gr(\gbf{\bullet}{R^G}) \to \gr(\gbf{\bullet}{R})$.
   This is a module-finite map of commutative rings, and $\gr(\gbf{\bullet}{R^G})$ is a finitely generated $\kk$-algebra of dimension $2 \dim(R^G)$.
\end{lemma}

\begin{proof}
   We are identifying $D_{R^G|\kk}$ with $(D_{R|\kk})^G \subseteq D_{R|\kk}$; under this identification, \eqref{eq: relating filtrations} is telling us that $\gbf{\bullet}{R^G}$ is the filtration on $D_{R^G|\kk}$ induced by $\gbf{\bullet}{R}$.
   Thus, we have a natural injective map $\gr(\gbf{\bullet}{R^G}) \to \gr(\gbf{\bullet}{R})$, which gives an isomorphism $\gr(\gbf{\bullet}{R^G})\cong \gr(\gbf{\bullet}{R})^G$.
   \Cref{prop: associated graded ring is polynomial ring} shows that $\gr(\gbf{\bullet}{R})$ is a polynomial ring of dimension $2\dim(R) = 2\dim(R^G)$, so the remaining claims follow from basic invariant theory of finitely generated commutative $\kk$-algebras.
\end{proof}

\begin{lemma}\label{lem:right-mod-gens-finite}
   Under \Cref{setup:finite-char0}, $D_{R|\kk}$ is a finitely generated right $D_{R^G|\kk}$-module.
   Moreover, there exist finitely many elements $\gamma_1,\dots,\gamma_\ell\in D_{R|\kk}$ and a positive integer~$v$ such that $\gbf{i}{R} \subseteq \gamma_1 \cdot \gbf{i+v}{R^G} + \cdots + \gamma_\ell \cdot \gbf{i+v}{R^G}$ for every $i \in \NN$.
\end{lemma}

\begin{proof}
   By \Cref{lem:invariants-fg-graded} and \Cref{prop:good-filtration First Noeth case}, $\gbf{\bullet}{R}$ is a good filtration on $D_{R|\kk}$ as a right $(D_{R^G|\kk},\gbf{\bullet}{R^G})$-module.
   Thus, $D_{R|\kk}$ is finitely generated as a right module over $D_{R^G|\kk}$, and $\gbf{\bullet}{R}$ is shift equivalent to a standard filtration $\{\gamma_1,\ldots,\gamma_\ell\}\gbf{\bullet}{R^G}$.
\end{proof}

Our main technical result in this section is the following.

\begin{theorem}\label{ThmRGLinearSimple}
   Let $G$ be a finite group that acts linearly on a polynomial ring $R$ over a field $\kk$ of characteristic zero.
   If $\gbf{\bullet}{R^G}$ is a generalized Bernstein filtration on $D_{R^G|\kk}$, then $(D_{R^G|\kk},\gbf{\bullet}{R^G})$ is linearly simple.	
\end{theorem}

\begin{proof}
   \Cref{lem:nopseudo} allows us to work under \Cref{setup:finite-char0}.
   Fix a finite set $\Sigma$ of generators for $R^G$ as a $\kk$-algebra, and choose $d$ so that $\Sigma\subseteq \gbf{d}{R^G}$. 
   Choose $C\in \PP$ such that for each $i\in \NN$ we have $\gbf{i}{R ^G}\subseteq D^{C i }_{R^G|\kk}$, as in \Cref{prop:compare-gbf-order}, and $1\in \gbf{Ci}{R} \cdot \delta \cdot \gbf{Ci}{R}$ for all $\delta\in \gbf{i}{R} \smallsetminus \{0\}$, as in \Cref{prop:bavula-poly}.
   Fix $\gamma_1,\dots,\gamma_\ell\in D_{R|\kk}$ and $v \in \PP$ with $\gbf{i}{R} \subseteq \gamma_1 \cdot \gbf{i+v}{R^G} + \cdots + \gamma_\ell \cdot \gbf{i+v}{R^G}$ for all $i$, as in \Cref{lem:right-mod-gens-finite}, and note that, by possibly increasing $v$, we may assume that $v > \ord(\gamma_j)$ and $\gamma_j \in \gbf{v}{R}$ for all $j$.

   We wish to show that there exists $K\in \PP$ such that for each $i\in \NN$ and each nonzero $\delta\in \gbf{i}{R^G}$, we have $1\in \gbf{Ki}{R^G}\cdot\delta \cdot \gbf{Ki}{R^G}$.
   Note, however, that since $\gbf{0}{R^G}=\kk$, this claim holds for $i=0$ with any $K$, so we shall focus on the case of a positive $i$.

   Fix a positive integer $i$ and a nonzero element $\delta\in \gbf{i}{R^G}$, and let $m = \ord(\delta)$.
   By \Cref{rmk: iterated commutators}, there exists a \emph{nonzero} $m$-fold commutator $[\cdots [[\delta,r_1],r_2],\ldots, r_m]$ with $r_1,\ldots, r_m \in \Sigma \subseteq \gbf{d}{R^G}$.
   Let $f$ denote this iterated commutator; then $f\in R^G\smallsetminus \{0\}$ and $f\in \gbf{md}{R^G}\cdot \delta \cdot \gbf{md}{R^G}$.
   Since $m \le Ci$ by our choice of $C$, we have 
   \begin{equation}\label{eq: where f is}
      f\in\gbf{Cdi}{R^G}\cdot \delta \cdot \gbf{Cdi}{R^G}.
   \end{equation}

   For $j \in \{1,\ldots,\ell\}$, write $\gamma_j^{(0)}=\gamma_j$ and recursively define $\gamma_j^{(k+1)}=[\gamma_j^{(k)},f]$.
   Note that $\gamma_j^{(v)}=0$ by our choice of $v$, and \Cref{commute-f} tells us that we can write
   \begin{equation}\label{eq: gammas}
      f^v \gamma_j = \gamma_j f^v + c_1 \gamma^{(1)}_j f^{v-1} + \cdots + c_{v-1} \gamma^{(v-1)}_j f
   \end{equation}
   for some integers $c_1,\ldots,c_{v-1}$.
   Since \eqref{eq: where f is} shows that $f\in \gbf{3Cdi}{R}$, and $\gamma_j \in \gbf{v}{R}$ by our choice of $v$, we have $\gamma_j^{(k)} \in \gbf{3Cdki+v}{R}$ for each $k$.
   Setting $\alpha = 3Cdv$, \eqref{eq: gammas} yields
   \begin{equation}
      f^v \gamma_j \in \gbf{\alpha i + v}{R} \cdot f.
      \label{eq: replacing f^v with f}
   \end{equation}

   As $f^v$ is a nonzero element of $\gbf{\alpha i}{R}$, our choice of $C$ implies that $1 \in \gbf{C\alpha i}{R} \cdot f^v \cdot \gbf{C\alpha i}{R}$.
   Writing $\gbf{C\alpha i}{R}\subseteq \gamma_1 \cdot \gbf{C\alpha i+v}{R^G} + \cdots + \gamma_\ell \cdot \gbf{C\alpha i+v}{R^G}$ and using \eqref{eq: replacing f^v with f} we see that
   $1 \in \gbf{(C+1)\alpha i + v}{R} \cdot f \cdot \gbf{C\alpha i+v}{R^G}$, so setting $\beta = (C+1)\alpha v$ we see that
   \begin{equation}
      1\in \gbf{\beta i}{R} \cdot f \cdot \gbf{\beta i}{R^G}.\label{eq: containment}
   \end{equation}

   Applying a Reynolds operator, we shall be able to replace the $\gbf{\beta i}{R}$ in \eqref{eq: containment} with $\gbf{\beta i}{R^G}$.
   Indeed, 
   the map $\rho:D_{R|\kk} \to (D_{R|\kk})^G = D_{R^G|\kk}$ given by $\delta \mapsto \frac{1}{|G|} \sum_{g\in G} g\cdot\delta$ maps $\gbf{i}{R}$ into $(\gbf{i}{R})^G = \gbf{i}{R^G}$ by \eqref{eq: relating filtrations}.
   This is a map of right $D_{R^G|\kk}$-modules, so applying $\rho$ to an equation expressing the containment in \eqref{eq: containment} we find that $1\in \gbf{\beta i}{R^G} \cdot f \cdot \gbf{\beta i}{R^G}$.
   To wrap up the proof, we invoke \eqref{eq: where f is} to conclude that
   \[1\in \gbf{\beta i}{R^G} \cdot \gbf{Cdi}{R^G} \cdot \delta \cdot \gbf{Cd i}{R^G} \cdot \gbf{\beta i}{R^G} \subseteq \gbf{(\beta+Cd) i}{R^G}  \cdot \delta \cdot \gbf{(\beta+Cd) i}{R^G}\]
   and stress that $\beta+Cd$ depends neither on $i$ nor on $\delta$.
\end{proof}

\begin{corollary}\label{Cor_invariant_Bernstein}
   Let $G$ be a finite group that acts linearly on a polynomial ring $R$ over a field $\kk$ of characteristic zero.
   Then the ring of invariants $R^G$ is a Bernstein $\kk$-algebra.
   In particular, every $D_{R^G|\kk}$-module satisfies Bernstein inequality with respect to any generalized Bernstein filtration.
   Furthermore, $R^G$, its localizations $R^G_f$ for $f\in R^G$, and its local cohomology modules $H^j_I(R^G)$ for $I\subseteq R ^G$ are all holonomic $D_{R^G|\kk}$-modules, and thus have finite length as $D_{R^G|\kk}$-modules. Likewise, for any nonzero element $f\in R^G$, $R^G_f(s) \fs$ is a holonomic $D_{R^G(s)|\kk(s)}$-module, and thus has finite length.
\end{corollary}

\begin{proof}
   This follows from \Cref{thm:dim-Bernstein-filtration,ThmRGLinearSimple}, using the fact that the differential signature of $R^G$ is positive \cite[Theorem~6.15]{BJNB}.
\end{proof}

We end this subsection with a couple of examples to illustrate that the main result of this section does not hold in the setting of rational singularities, even for hypersurfaces.

\begin{example}\label{eg51}
   Let $\kk$ be a field of characteristic zero, and
   \[	R =\frac{\kk[ w, x, y, z]}{(w^3+x^3+y^3+z^3)}.
   \]
   This is a standard graded hypersurface domain with rational singularities for which $D_{R|\kk}$ has no elements of negative degree \cite[Theorem~1.2]{Mallory}.
   Consequently, $(D_{R|\kk},\gbf{\bullet}{R})$ is not linearly simple, and thus $R$ is not a Bernstein algebra.
   Moreover, the maximal ideal $(w,x,y,z)$ is a proper nontrivial $D_{R|\kk}$-submodule of $R$, whence $R$ is not a simple $D_{R|\kk}$-module.
   For this ring, the residue field $R/(w,x,y,z)\cong \kk$ is a $D_{R|\kk}$-module; any filtration on this module has dimension zero.
\end{example}

\begin{example}\label{eg52}
   Let $\kk$ be a field of characteristic zero, and
   \[	S =\frac{\kk[s, t, u, v, w, x, y, z]}{(su^2x^2 + sv^2y^2 + tuxvy + tw^2z^2)}.
   \]
   This is a standard graded hypersurface domain with rational singularities that has a local cohomology module with infinitely many associated primes \cite[Theorem~5.1]{SinghSwanson}, and hence not a Bernstein algebra by \Cref{prop: finite ass}.
\end{example}

\subsection{Holonomicity and differential direct summands} \label{Sec52}

Let $R$ be a polynomial ring over a field $\kk$ of characteristic zero and let $G$ be a finite group acting linearly on $R$. 
The ring of invariants $R^G $ is then a direct summand of $R$.
Namely, the inclusion $R^G \hookrightarrow R$ has a splitting $\beta: R \to R^G$ given by the Reynolds operator.
This is one of the main examples in the theory of \emph{differential direct summands}  \cite{AMHNB,square}.
We briefly recall the basics in our setting and refer to loc.~cit.\ for more insight.

First, notice that for any $\delta\in D_{R|\kk}$, the map $\beta \circ \delta|_{R^G} \colon R^G \to R^G$ is an element of $D_{R^G|\kk}$.
That is, we have the following diagram:
\[
   \xymatrix{
            R^G\, \ar@{.>}[d] \ar@{^{(}->}[r] & R\ar[d]^{\delta}\\
            R^G & R \ar[l]_{\beta} }
\]

We say that a $D_{R^G|\kk}$-module $M$ is a \emph{differential direct summand} of a $D_{R|\kk}$-module $N$ if $M\subseteq N$ and there exists an $R^G$-linear splitting $\Theta\colon N\to M$, called a \emph{differential splitting}, such that
\[  \Theta(\delta \act v) = (\beta \circ \delta|_{R^G}) \act v \]
for every $\delta\in D_{R|\kk}$ and $v\in M$, where the action on the left-hand side is the $D_{R|\kk}$-action, considering $v$ as an element of $N$, and the action on the right-hand side is the $D_{R^G|\kk}$-action.
Among the properties that these modules satisfy we have ${\rm length}_{D_{R^G|\kk}}(M) \leq {\rm length}_{D_{R|\kk}}(N)$ \cite[Proposition~3.4]{AMHNB}.
In particular, any differential direct summand of a holonomic $D_{R|\kk}$-module has finite length.

The main examples of differential direct summands of holonomic modules are the rings $R^G \subseteq R$ themselves, the localizations $R^G_f \subseteq R_f$ at elements $f\in R^G$, and the local cohomology modules $H^i_I(R^G) \subseteq H^i_{IR}(R)$ at ideals $I\subseteq R^G$. 

The goal of this subsection is to prove that the holonomic $D_{R^G|\kk}$-modules are precisely the differential direct summands of holonomic $D_{R|\kk}$-modules.

\begin{theorem} \label{holonomic_direct_summand}
   Let $G$ be a finite group acting linearly on a polynomial ring $R$ over a field $\kk$ of characteristic zero, and suppose that $G$ contains no pseudoreflections.
   Then a $D_{R^G|\kk}$-module $M$ is holonomic if and only if it is a differential direct summand of a holonomic $D_{R|\kk}$-module with respect to the splitting $\beta:R\to R^G$ given by the Reynolds operator.
\end{theorem}

\begin{proof}
   As in \Cref{setup:finite-char0}, we can identify $D_{R^G|\kk}$ with $(D_{R|\kk})^G \subseteq D_{R|\kk}$, and choose generalized Bernstein filtrations so that $\gbf{\bullet}{R^G} = \gbf{\bullet}{R} \cap D_{R^G|\kk}$.
   Let $M$ be a holonomic left $D_{R^G|\kk}$-module, and $\filtration{G}{\bullet}{}$ a good filtration on $M$ compatible with $\gbf{\bullet}{R^G}$.
   As in \Cref{lem:right-mod-gens-finite}, take $\gamma_1,\dots,\gamma_\ell\in D_{R|\kk}$ and a positive integer $v$ such that $\gbf{i}{R} \subseteq \gamma_1 \cdot \gbf{i+v}{R^G} + \cdots + \gamma_\ell \cdot \gbf{i+v}{R^G}$ for all $i$.
   By possibly replacing $v$ with a larger value, we may assume that $\gbf{i}{R}\cdot \gamma_{t} \subseteq \gamma_1 \cdot \gbf{i+v}{R^G} + \cdots + \gamma_\ell \cdot \gbf{i+v}{R^G}$ for each $i$ and $t$.

   Let $N= D_{R|\kk} \otimes_{D_{R^G|\kk}} M$, and for each $i \in \NN$ let $\filtration{H}{i}{}$ be the $\kk$-subspace of $N$ spanned by $\{\gamma_1 \otimes \filtration{G}{vi}{},\dots, \gamma_\ell \otimes \filtration{G}{vi}{}\}$.
   Then $\filtration{H}{\bullet}{}$ is finite dimensional, ascending, and, as the $\gamma_t$ generate $D_{R|\kk}$ over $D_{R^G|\kk}$, exhaustive.
   Note that if $m\in \filtration{G}{vj}{}$ and $i$ is positive, then
   \[
      \gbf{i}{R} (\gamma_t\otimes m )\subseteq
      \sum_{k=1}^\ell \gamma_k \gbf{i+v}{R^G} \otimes m =
      \sum_{k=1}^\ell \gamma_k \otimes \gbf{i+v}{R^G} m \subseteq
      \sum_{k=1}^\ell \gamma_k \otimes \filtration{G}{vi+vj}{}
   \]
   which shows that $\gbf{i}{R} \filtration{H}{j}{} \subseteq \filtration{H}{i+j}{}$.
   The same holds for $i=0$, since $\gbf{0}{R}=\kk$.
   Thus, $\filtration{H}{\bullet}{}$ is a filtration compatible with $\gbf{\bullet}{R}$.
   It follows easily from the fact that $M$ is holonomic that $\filtration{H}{\bullet}{}$ has dimension $\dim(R)$ and finite multiplicity, so $N$ is holonomic.

   Now we check that $M$ is a differential direct summand of $N$.
   The averaging map $\rho:D_{R|\kk} \to (D_{R|\kk})^G$ given by $\delta \mapsto \frac{1}{|G|} \sum_{g\in G} g \cdot \delta$ induces a $D_{R^G|\kk}$-linear splitting $\Theta$ of the natural map $M\to N$; in particular, $M$ injects into $N$.
   As $(\beta \circ \delta)|_{R^G} = \rho(\delta)$ for each $\delta \in D_{R|\kk}$, the maps $\beta$ and $\Theta$ induce a differential direct summand structure.

   For the converse, let $M$ be a differential direct summand of $N$ with respect to~$\beta$ and some $\Theta$.
   It follows from 
   \eqref{eq: relating filtrations}
   that $\gbf{i}{R^G} = \rho(\gbf{i}{R})$, so any element of $\gbf{i}{R^G}$ can be realized as $(\beta\circ \delta)|_{R^G}$ for some $\delta\in \gbf{i}{R^G}$.
   It then follows from the differential direct summand condition that the image under $\Theta$ of a good filtration for $N$ is compatible with the generalized Bernstein filtration $\gbf{\bullet}{R^G}$ on $D_{R^G|\kk}$, and has dimension at most $\dim(R^G)$ and finite multiplicity. 
\end{proof}

\subsection{Comparison with van den Essen's notion of holonomicity}

We briefly discuss the notion of holonomicity of van den Essen \cite{VdE86} arising from Gabber's theorem on involutivity of singular supports \cite{Gabber}. We refer the reader to these sources for further details.

Let $\kk$ be a field of characteristic zero and $A$ a not necessarily commutative $\kk$-algebra with a filtration $\filtration{F}{\bullet}{}$ by $\kk$-vector spaces that are not necessarily finite dimensional---that is, a collection of subspaces that satisfies all of the axioms of a filtration in \Cref{notation:filtration} except finite dimensionality. If $\gr(\filtration{F}{\bullet}{})$ is commutative, then it admits the structure of a \emph{Poisson algebra}, i.e., a commutative ring with a Lie bracket 
\[ \{ - , - \} \ : \ \gr(\filtration{F}{\bullet}{})\times \gr(\filtration{F}{\bullet}{}) \to \gr(\filtration{F}{\bullet}{}) \] such that for each $g\in \gr(\filtration{F}{\bullet}{})$, the map
\[ \{ g, - \} \ : \ \gr(\filtration{F}{\bullet}{}) \to \gr(\filtration{F}{\bullet}{})\] 
is a derivation; this Lie bracket is called a Poisson bracket. The Poisson bracket on $\gr(\filtration{F}{\bullet}{})$ is defined as follows: For $g \in \gr(\filtration{F}{\bullet}{})_a$ and $h\in \gr(\filtration{F}{\bullet}{})_b$, write $g= \delta+ \filtration{F}{a-1}{}$ and $h= \eta + \filtration{F}{b-1}{}$, and set
\[ \{ g ,h \} = [ \delta, \eta ] + \filtration{F}{a+b-2}{} \in \gr(\filtration{F}{\bullet}{})_{a+b-1}.\]

Suppose that $\gr(\filtration{F}{\bullet}{})$ is a finitely generated $\kk$-algebra. Let $M$ be a finitely generated $A$-module and $\filtration{G}{\bullet}{}$ be a good filtration of $M$ with respect to the filtration~$\filtration{F}{\bullet}{}$. One then has that the ideal \[ I = \ann_{\gr(\filtration{F}{\bullet}{})}(\gr(\filtration{G}{\bullet}{}))\] satisfies $\{I,I\} \subseteq I$; that is, $I$ is \emph{involutive}. A famous result of Gabber \cite{Gabber} shows that in this setting  $\sqrt{I}$ is also involutive. When $R$ is a polynomial ring over $\kk$, one can then deduce Bernstein's inequality by relating the Poisson structure on the associated graded algebra of $D_{R|\kk}$ via the order filtration with symplectic geometry and showing that any involutive ideal has height bounded by the dimension of $R$.

Following Gabber's theorem, van den Essen posed a definition for a notion of holonomicity in rings with a filtration $\filtration{F}{\bullet}{}$ for which the associated graded ring is a finitely generated commutative algebra over a field $\kk$. Namely, he defines a nonzero module $M$ over such a ring to be holonomic if, for a good filtration of $M$, the height of every minimal prime of the annihilator of the associated graded module of $M$ is equal to the maximal height of a graded involutive prime ideal in the associated graded algebra $\gr(\filtration{F}{\bullet}{})$.

The purpose of this section is to note that for invariant rings of finite groups in characteristic zero, there may be no nonzero modules that are holonomic in the sense of van den Essen.

\begin{example} \label{example:noVDE} Let $R$ be a polynomial ring over a field $\kk$, and $G$ a finite group acting linearly on $R^G$ with no pseudoreflections. Note that $D^1_{R^G|\kk}$ has no elements of negative degree; this follows from the equality $D^1_{R^G|\kk}=(D^1_{R|\kk})^G$ and the observation that the action of $G$ on the associated graded ring of $D_{R|\kk}$ again has no pseudoreflections. Let $A$ be the associated graded ring of $D_{R^G|\kk}$ with respect to the order filtration. Note that $A$ is bigraded with a first grading coming from the grading on $D_{R^G|\kk}$ and a second grading arising from the associated graded structure; write $A_{i,j}$ for the bigraded pieces. Note that $A_{0,0}\cong\kk$. 

   We claim that $A_+ \coloneqq \bigoplus_{i\in \NN} A_{i,0} \oplus \bigoplus_{i \in \ZZ, j>0} A_{i,j}$ is an involutive ideal.
   It suffices to check that for homogeneous elements $a,b\in A_+$ we have $\{a,b\}\in A_+$. Let $a\in A_{i,j}$ and $b\in A_{k,\ell}$, so $\{a,b\}\in A_{i+k,j+\ell-1}$. Thus it suffices to deal with the case $j=1, \ell=0$. For $\{a,b\}\notin A_+$, we must have $i+k=0$. However, $A_{k,0}=0$ for $k<0$ and $A_{i,1} = 0$ for $i<0$ since there are no operators of order $1$ and negative degree. But if $b\in A_{0,0}$, we have $\{a,b\}=0\in A_+$.

Thus, for a nonzero $D_{R^G|\kk}$-module $M$ to be holonomic in the sense of van den Essen, the annihilator of an associated graded must be primary to the homogeneous maximal ideal of $A$, so $M$ must be a finite-dimensional $\kk$-vector space. The existence of such a module is contradicted by \Cref{Cor_invariant_Bernstein}.
\end{example}

\section{Strongly $F$-regular finitely generated graded algebras with FFRT}

In this section, we focus on commutative rings of prime characteristic, more precisely on the class of rings with finite $F$-representation type, introduced by Smith and Van den Bergh \cite{SVDB}.
We prove that if such a ring is strongly $F$-regular, then it is a Bernstein algebra, and thus Bernstein's inequality is satisfied, and the dimension of $D_{R|\kk}$ is twice the dimension of $R$.

To discuss finite $F$-representation type, we need a class of rings for which the analogue of the Krull--Schmidt Theorem holds, and over which the Frobenius map is finite.
Throughout this section we work in the following setting.

\begin{setup} \label{Notation_graded_prime}
   Let $\kk$ be a perfect field of prime characteristic $p$, and $R$ a \cga, as in \Cref{conv: commutative graded ring}, that is a domain.
   Set $\m=\bigoplus_{i>0} R_i$, $n=\dim_\kk \m/\m^2$, and $w=\max\{t\in\NN\;|\;[ \m/\m^2]_t\neq 0\}$.
\end{setup}

In the remarks that follow we introduce a suitable category of modules to work with, and make some considerations about the structure of rings of endomorphisms of such modules---especially concerning maximal homogeneous two-sided ideals.

\begin{remark}[A Krull--Schmidt category]
   \label{rem: KS category}
   If $R$ is as in \Cref{Notation_graded_prime}, then the category $\cat$ of finitely generated $\QQ$-graded $R$-modules and homogeneous homomorphisms is a Krull--Schmidt category---that is, every nonzero object in $\cat$ can be expressed uniquely, up to order and isomorphism, as a direct sum of indecomposable objects.
   Indeed, given objects $M$ and $N$ in $\cat$, suppose $M$ is generated by homogeneous elements of degrees $d_1,\ldots,d_m$, and let us denote the $\kk$-vector space of homogeneous homomorphisms $M\to N$ by $\HomogHom{R}{M}{N}$.
   The map $\HomogHom{R}{M}{N} \to \bigoplus_{i=1}^m\Hom_\kk(M_{d_i},N_{d_i})$ given by restriction is an injective $\kk$-linear map; its codomain is a finite-dimensional $\kk$-vector space, and thus so is its domain, $\HomogHom{R}{M}{N}$.
   This implies that $\cat$  is a Krull--Schmidt category \cite[Corollary to Lemma~3, Theorem~1]{Atiyah}.
\end{remark}

\begin{remark}[On the structure of rings of endomorphisms]
   \label{RemMaximalEnd}
   For every $\QQ$-graded $R$-module $M$, the module of (not necessarily graded) $R$-endomorphisms $\Lambda = \End_{R}(M)$ is a ring with unity, which is usually not commutative.
   If $M$ is finitely generated, then $\Lambda$ is a $\QQ$-graded ring, with the decomposition $\Lambda = \bigoplus_{i\in \QQ} \Lambda_{i}$, where $\Lambda_i = \GrEnd{R}{M}{i}$ consists of the graded endomorphisms of degree $i$ for each $i\in \QQ$ \cite[Corollary~I.2.11]{GradedRingTheory}.
   If, in addition, $M$ is indecomposable in the category $\cat$ of \Cref{rem: KS category}, then the degree~0 component $\Lambda_0 = \HomogEnd{R}{M}$ is a local ring, that is, the set of all nonunits is a two-sided ideal, which is necessarily the unique maximal (right, left, and two-sided) ideal of $\Lambda_0$ \cite[Lemma~7]{Atiyah}.
   This, in turn, implies that $\Lambda$ has a unique maximal homogeneous (left, right, and two-sided) ideal---namely, the ideal generated by all the homogeneous nonunits, or equivalently, the set of all elements whose homogeneous components are all nonunits \cite[Theorem~2.5]{Li}.

   If $N\cong M^\alpha$, where $M$ is indecomposable in $\cat$, then $\End_{R}(N)$ can be identified with $\End_{R}(M)^{\alpha\times \alpha},$ the ring of $\alpha\times \alpha$ matrices with entries in $\End_{R}(M)$.
   Since all homogeneous two-sided ideals of $\End_{R}(M)^{\alpha \times \alpha}$ are of the form $\mathfrak{A}^{\alpha \times \alpha}$, for $\mathfrak{A}$ a homogeneous two-sided ideal of $\End_{R}(M)$, it follows that $\End_{R}(N)$ also has a unique maximal homogeneous two-sided ideal, namely the set of all matrices whose entries decompose as sums of homogeneous nonunits.

   Note that if $N'$ is also a direct sum of $\alpha$ copies of $M$, but with rational degree shifts applied to the various summands, then $\End_R(N')$ agrees with $\End_R(N)$ as an ungraded ring, but their gradings may differ by shifts in the off-diagonal entries.
   The endomorphism ring $\End_R(N')$ again has a unique maximal homogeneous two-sided ideal---in fact the same one as $\End_R(N)$, but with possible changes in the grading. 

   Generalizing further, suppose now that
   \[N=M^{\alpha_1}_1\oplus \cdots \oplus M^{\alpha_\ell}_\ell\]
   where $M_1,\ldots, M_\ell$ are indecomposable in $\cat$, and $M_i$ is not graded-isomorphic to $M_j(d)$ for any $d\in \QQ$
and $i\neq j$.
   An endomorphism $\phi \in \End_{R}(N)$ can be viewed as a block matrix $(\phi_{ij})_{1\le i,j \le \ell}$, where $\phi_{ij} \in \Hom_{R}(M_j,M_i)^{\alpha_i \times \alpha_j}$.
   The endomorphism ring $\End_{R}(N)$ has $\ell$ maximal homogeneous two-sided ideals, $\N_1,\ldots,\N_\ell$, where each~$\N_k$ consists of all~$(\phi_{ij})$ such that every homogeneous component of every entry of~$\phi_{kk}$ is a nonunit.
   Although this appears to be a well-known fact to the specialists, for lack of an appropriate reference, we provide the proof below, in \Cref{prop: maximal ideals of End(N)}.
   Once again, the considerations just made are unaffected by degree shifts in the various components of $N$. 
\end{remark}

\begin{proposition}
   \label{prop: maximal ideals of End(N)}
   Suppose $N = M_1^{\alpha_1}\oplus \cdots \oplus M_\ell^{\alpha_\ell}$, as in \Cref{RemMaximalEnd}.
   As in that remark, we view the endomorphisms of $N$ as block matrices $(\phi_{ij})_{1\le i,j\le \ell}$, with $\phi_{ij} \in \Hom_{R}(M_j,M_i)^{\alpha_i \times \alpha_j}$. 
   For each $k=1,\ldots,\ell$, let $\n_k$ be the unique maximal homogeneous two-sided ideal of $\End(M_k)^{\alpha_k\times\alpha_k}$, and let $\N_k$ be the subset of $\End_R(N)$ consisting of block matrices $(\phi_{ij})$ with $\phi_{kk} \in \n_k$ \textup(that is, every homogeneous component of every entry of $\phi_{kk}$ is a nonunit\textup).
   Then $\N_1,\ldots,\N_\ell$ are all of the maximal homogeneous two-sided ideals of $\End_R(N)$. 
\end{proposition}

\begin{proof}
   We first show that $\N_1$ is a homogeneous two-sided ideal.
  We have that $\N_1$ is a $\QQ$-graded additive subgroup of $\End_R(N)$, and the multiplicative properties of an ideal can thus be verified using homogeneous elements.   
   Take homogeneous elements $\Phi = (\phi_{ij}) \in \N_1$ and $\Psi = (\psi_{ij}) \in \End_R(N)$, and suppose $\Phi \Psi \notin \N_1$, so the upper left block of that product, $\sum_{k=1}^\ell \phi_{1k}\psi_{k1}$, does not lie in $\n_1$.
   As $\phi_{11}\in \n_1$, this implies that $\phi_{1k}\psi_{k1} \notin \n_1$ for some $k\in \{2,\ldots,\ell\}$.
   This, in turn, implies that there are entries $\delta: M_k \to M_1$ in $\phi_{1k}$, and $\gamma: M_1 \to M_k$ in $\psi_{k1}$, such that $u\coloneqq \delta\gamma: M_1 \to M_1$ is a unit.
   Thus, $\gamma$ is injective, and $u^{-1}\delta$ is a splitting of $\gamma$.
   As $M_k$ is indecomposable, $\gamma$ must be surjective as well---but this contradicts our assumption that $M_1$ is not graded-isomorphic to $M_k(d)$ for any $d\in \QQ$.
   This shows that $\Phi\Psi \in \N_1$, and the proofs that $\Psi\Phi \in \N_1$, and that the other $\N_k$ are homogeneous two-sided ideals, follow similar steps.

   We now show that each proper homogeneous two-sided ideal $\A$ is contained in some $\N_k$.
   For each $k=1,\ldots,\ell$, let $\ideala_k$ be the image of $\A$ under the projection $\End_R(N) \twoheadrightarrow \End_R(M_k)^{\alpha_k\times\alpha_k}$, and note that $\ideala_k$ is a homogeneous two-sided ideal of $\End_R(M_k)^{\alpha_k\times\alpha_k}$.
   If $\ideala_k$ is not proper, then $\A$ contains an element $\Phi=(\phi_{ij})$ where $\phi_{kk} = \mathbf{1}$, the identity of $\End_R(M_k)^{\alpha_k\times\alpha_k}$.
   Multiplying $\Phi$ on both sides by the element $\Psi_k = (\psi_{ij}) \in \End_R(N)$ with $\psi_{kk} = \mathbf{1}$ and zeros elsewhere, we see that $\Psi_k$ itself lies in $\A$.
   If none of the $\ideala_k$ were proper, this reasoning would show that $\A$ contains $\sum_{k=1}^\ell\Psi_k$, the identity of $\End_R(N)$, contradicting its properness.
   Thus, $\ideala_k$ is proper for some $k$, whence $\ideala_k \subseteq \n_k$ and $\A \subseteq \N_k$.
\end{proof}

Next, we recall definitions that involve the structure of rings regarding the Frobenius homomorphism.
To this end, note that the ring of $p^e$-th roots $R^{1/p^e}$ is a $\QQ$-graded (in fact, $\frac{1}{p^e}\NN$-graded) $R$-module, with $(R^{1/p^e})_{i/p^e}=(R_i)^{1/p^e}$, and under this grading the inclusion $R\hookrightarrow R^{1/p^e}$ is homogeneous.
Note also that, because $R$ is a finitely generated commutative algebra over a perfect field, $R$ is $F$-finite, that is, $R^{1/p^e}$ is a finitely generated $R$-module.

\begin{definition}[{\cite[Definition~3.1.1]{SVDB}}]
   Let $R$ be as in \Cref{Notation_graded_prime}.
   The ring $R$ has \emph{finite $F$-representation type}, \emph{FFRT} for short, if there exists a finite set of indecomposable finitely generated $\QQ$-graded $R$-modules, $M_1,\ldots,M_\ell$, such that $R^{1/p^e}$ is isomorphic to a direct sum of finitely many copies of the $M_i$, with possible rational degree shifts, for every $e\in\NN$.
\end{definition}

Rings of finite $F$-representation type include invariant rings of linearly reductive groups, simple hypersurface singularities, and the homogeneous coordinate ring of $\mathrm{Gr}(2,n)$ \cite{SVDB,RSV1,RSV2}.

\begin{definition}
   Let $R$ be as in \Cref{Notation_graded_prime}.
   The \emph{splitting ideals} of $R$ are defined by
   \[I_e(R)=\{r\in R \;|\; \varphi(r^{1/p^e})\in\m\text{ for each } \varphi\in \Hom_R(R^{1/p^e}, R)\}\]
   for each $e\in\NN$.
   The ring $R$ is \emph{$F$-pure} if $I_e(R)\neq R$ for some (equivalently, all) $e\geq 1$, and \emph{strongly $F$-regular} if $\bigcap_{e\in\NN} I_e(R)=0$.
\end{definition}

We point out that these definitions for $F$-purity and strong $F$-regularity are not the usual ones found in the literature. 
Hochster and Roberts \cite{HR1976} were the first to study $F$-purity, while the notion of strong $F$-regularity was originally defined by Hochster and Huneke \cite{HoHuStrong}.
The splitting ideals $I_e(R)$ first appeared in the work of Aberbach and Enescu \cite{AE} and Yao \cite{Yao06}, although in a different formulation; our formulation appears in the work of Tucker \cite{Tucker12}.

The following result was proved by Polstra and Tucker in the case of complete local rings, and extends to our setting through localization and completion at $\m$.

\begin{lemma}[{\cite[Lemma~5.2, second proof of Theorem~5.1]{PT}}]\label{LemmaUniformPowers}
   Let $R$ be as in \Cref{Notation_graded_prime}.
   If $R$ is strongly $F$-regular, then there exists $a\in\NN$ such that
   \[
      \pushQED{\qed}
      I_{a+e}(R)\subseteq \m^{p^e} \quad \text{for all $e\in \NN$}.
      \qedhere
      \popQED
    \]
\end{lemma}

Toward the proof of the main technical result in this section, \Cref{ThmFFRTLinearSimple}, it is convenient to isolate the following elementary lemma.  

\begin{lemma}
   \label{lem: LT generates unit ideal implies element generates it as well}
   Let $A$ be a $\ZZ$-graded ring, not necessarily commutative, with $A_i = 0$ for all $i \ll 0$.
   Let $\delta$ be a nonzero element of $A$, and let $\lt$ be its homogeneous component of largest degree.
   Suppose $A\lt A = A$.
   Then every nonzero homogeneous element $\gamma$ of $A$ can be expressed as a sum of elements of the form $\alpha \delta\beta$, where $\alpha,\beta \in A$ are homogeneous and $\deg(\alpha\lt\beta) \le \deg(\gamma)$.
   In particular, $A \delta A = A$ as well.
\end{lemma}

\begin{proof}
   We use induction on the degree of $\gamma$, noting that the base of induction holds (vacuously) because of our assumption on $A$.
   Let $\gamma \in A$ be a nonzero homogeneous element, and suppose the claim holds for like elements of smaller degree.
   Since $A\lt A = A$, we can write $\gamma = \sum_{i=1}^r \alpha_i \lt \beta_i$ for some $\alpha_i,\beta_i \in A$.
   By expanding and comparing homogeneous components, we may assume that $\alpha_i$ and $\beta_i$ are homogeneous and $\deg(\alpha_i\lt \beta_i) = \deg(\gamma)$ for each $i$.
   Now write $\delta = \lt + \delta_1 + \cdots + \delta_s$, where each $\delta_j$ is homogeneous with $\deg(\delta_j) < \deg(\lt)$.
   Then 
   \[\gamma = \sum_{i = 1}^r \alpha_i \delta \beta_i - \sum_{i=1}^r\sum_{j=1}^s \alpha_i \delta_j \beta_i\]
   where each nonzero term in the double sum is homogeneous of degree less than $\deg(\gamma)$, so applying the induction hypothesis to those terms gives us the result.
\end{proof}

We are now ready to prove the main result in this section.

\begin{theorem}\label{ThmFFRTLinearSimple}
   Let $R$ be as in \Cref{Notation_graded_prime}, and suppose that $R$ is strongly $F$-regular with FFRT.
   If $\gbf{\bullet}{R}$ is a generalized Bernstein filtration on $D_{R|\kk}$, then $(D_{R|\kk},\gbf{\bullet}{R})$ is linearly simple.
\end{theorem}

\begin{proof}
   Recall from \Cref{Notation_graded_prime} that $n=\dim_\kk \m/\m^2$ and $w =\max\{t\in\NN\;|\;[ \m/\m^2]_t\neq 0\}$.
   \Cref{prop: linear simplicity is independent of slope} allows us to assume that the slope of $\gbf{\bullet}{R}$ is $2w$.
   Fix $Z\in \PP$    
   such that $\gbf{i}{R}\subseteq D_{R|\kk}^{Z i}$ for each~$i$, as in \Cref{prop:compare-gbf-order}, and $a \in \PP$ such that $I_{a+e}(R)\subseteq \m^{p^e}$ for every $e\in\NN$, as in \Cref{LemmaUniformPowers}.
   
   As $R$ has FFRT, there exist finitely generated $\QQ$-graded $R$-modules $M_1,\ldots,M_\ell$ that are indecomposable in the category $\cat$ of \Cref{rem: KS category}, with $M_i$ not graded-isomorphic to $M_j(d)$ for any $i\ne j$ and $d\in \QQ$,
   and such that each $R^{1/p^e}$ is a direct sum of copies of the $M_i$, with possible rational degree shifts, and every $M_i$ is a direct summand of some $R^{1/p^{e_i}}$, again with a possible degree shift.
   Set $b = \max\{e_1,\ldots,e_\ell\}$.
   As $R$ is strongly $F$-regular, it is in particular $F$-split, so each $R^{1/p^{e_i}}$ is an $R$-direct summand of $R^{1/p^b}$, and the decomposition of $R^{1/p^b}$ as a direct sum of indecomposable $\QQ$-graded $R$-modules contains \emph{all} the $M_i$, up to degree shift, by the Krull--Schmidt Theorem.
   
   Let $i \in \NN$ and $\delta\in \gbf{i}{R}\smallsetminus \{0\}$.
   We wish to show the existence of $C \in \PP$, independent of~$i$ or $\delta$, such that $1\in \gbf{C i}{R}\cdot\delta\cdot \gbf{C i}{R}$.
   As in the proof of \Cref{ThmRGLinearSimple}, we assume that $i > 0$.

   Let $\lt$ be the homogeneous component of $\delta$ of largest degree.
   Then $\lt \in \gbf{i}{R} \subseteq D_{R|\kk}^{Zi}$, and hence $-wZ i \le \deg (\lt)\leq i$.
   Because $R$ can be generated by elements of degree at most~$w$, the nonzero operator $\lt$ must act nontrivially on degree $\le w \ord(\lt)$, that is, there exists $f\in R$ homogeneous with $\deg(f)\leq w\ord(\lt)\leq w Z i$ such that $g\coloneqq\lt(f)\neq 0$.
   We note that $\deg(g)=\deg(f)+\deg(\lt)\leq wZ i +i$.
   In particular, $f$ and $g$ are not in $\m^{(wZ +2) i}$, and thus if we set $c =\lceil \log_p ((wZ +2) i )\rceil$, we see that $f$ and $g$ do not lie in $I_{a+c}(R) \subseteq \m^{p^c}$.
   Equivalently, $f^{1/p^b}$ and $g^{1/p^b}$ do not lie in $I_{a+c}(R^{1/p^b})$, and putting $e=a+b+c$, we deduce that $R^{1/p^b} f^{1/p^e}$ and $R^{1/p^b} g^{1/p^e}$ are free direct summands of $R^{1/p^e}$.
   
   As $D_{R|\kk}^{Z i}\subseteq D_{R|\kk}^{(c)}$ by \eqref{eq: comparing order and level filtrations}, $\lt$ is $R^{p^c}$-linear, and the map $\lt^{\,1/p^e}:R^{1/p^e}\to R^{1/p^e}$ defined as in \eqref{eq: definition of roots of maps} is $R^{1/p^{a+b}}$-linear, and thus $R$-linear.
   This map induces a bijection between the free summands $R^{1/p^b}f^{1/p^e}$ and $R^{1/p^b}g^{1/p^e}$.
   Given our choice of $b$, and viewing $\lt^{\,1/p^e}$ as a block matrix as in \Cref{RemMaximalEnd}, it follows that for each $j$, the block of $\lt^{\,1/p^e}$ with entries in $\End_R(M_j)$ must contain at least one (homogeneous) unit.
   This shows that $\lt^{\,1/p^e}$ does not lie in any maximal homogeneous two-sided ideal of $\End_{R}(R^{1/p^e})$.
   Equivalently, $\lt$ does not lie in any maximal homogeneous two-sided ideal of $\End_{R^{ p^e }} (R)=D_{R|\kk}^{(e)}$, and hence the two-sided homogeneous ideal it generates must be the unit ideal, that is,  
   \[
      D_{R|\kk}^{(e)}\cdot \lt \cdot D_{R|\kk}^{(e)}=D_{R|\kk}^{(e)}.
   \]
   Since $D_{R|\kk}^{(e)} \subseteq D_{R|\kk}^{np^e}$ by \eqref{eq: comparing order and level filtrations 2}, $D_{R|\kk}^{(e)}$ is zero in degree $< - wnp^e$.
   \Cref{lem: LT generates unit ideal implies element generates it as well} then shows the existence of $\alpha_j,\beta_j \in D_{R|\kk}^{(e)}$ homogeneous with $\deg(\alpha_j \lt \beta_j) \le 0$ such that
   \begin{equation}
      \label{eq: 1 =}
      1 = \sum_{j = 1}^m \alpha_j \delta \beta_j.
   \end{equation}

   Recalling that $c =\lceil \log_p ((wZ +2) i )\rceil$ and $e=a+b+c$, we now observe that
   \[
      D_{R|\kk}^{(e)}\subseteq D_{R|\kk}^{np^e}\subseteq D_{R|\kk}^{np^{a+b+1} (wZ +2)i }
   \]
   and set $Y =np^{a+b+1}(wZ +2)$. 
   The $\alpha_j$ and $\beta_j$ lie in $D_{R|\kk}^{Y i}$, so their degrees are bounded below by $-wY i $.
   As $\deg(\alpha_j \lt \beta_j) \le 0$, we see that
   \[
      \deg(\alpha_j)\le -\deg(\lt)-\deg(\beta_j)\leq wZ i+wY i
   \]
   and so
   \[
      \deg(\alpha_j)+ 2w\ord(\alpha_j)\leq wZ i + wY i+2w Y i.
   \]
   We have just shown that $\alpha_j\in \gbf{C i}{R}$, where $C =w(Z +3Y)$.
   Likewise, $\beta_j\in \gbf{C i}{R}$, and \eqref{eq: 1 =} tells us that $1\in \gbf{C i}{R}\cdot\delta\cdot\gbf{C i}{R}$, which concludes the proof.
\end{proof}

\begin{corollary} \label{Cor_FFRT_Bernstein}
   Let $R$ be as in \Cref{Notation_graded_prime}, and suppose that $R$ is strongly $F$-regular with FFRT.
   Then $R$ is a Bernstein algebra.
   In particular, every $D_{R|\kk}$-module satisfies Bernstein inequality with respect to any generalized Bernstein filtration.
   Furthermore, $R$, its localizations $R_f$ for $f\in R$, and its local cohomology modules $H^j_I(R)$ for $I\subseteq R$ are holonomic $D_{R|\kk}$-modules, and hence have finite length as $D_{R|\kk}$-modules.
 \end{corollary}

\begin{proof}
   This follows from \Cref{thm:dim-Bernstein-filtration,ThmFFRTLinearSimple}, using the fact that the differential signature of a strongly $F$-regular $\kk$-algebra is positive \cite[Theorem~5.17]{BJNB}.
\end{proof}

We end with a couple of examples to demonstrate the necessity of the hypotheses.

\begin{example}\label{eg61}
   Let $\kk$ be a field of positive characteristic, and $R=\displaystyle \frac{\kk[x,y]}{(xy)}$.
   Then $R$ has FFRT, but is not strongly $F$-regular.
   The ring $R$ is not a Bernstein algebra since $(D_{R|\kk},\gbf{\bullet}{R})$ is not linearly simple; moreover, $R$ is not a simple $D_{R|\kk}$-module.
   For this ring, the residue field $R/(x,y)\cong \kk$ is a $D_{R|\kk}$-module; any filtration on this module has dimension zero.
\end{example}

\begin{example}\label{eg62}
   Let $\kk$ be a field of positive characteristic, and
   \[	S =\frac{\kk[s, t, u, v, w, x, y, z]}{(su^2x^2 + sv^2y^2 + tuxvy + tw^2z^2)}.
   \]
   This is a strongly $F$-regular standard graded domain that has a local cohomology module with infinitely many associated primes \cite[Theorem~5.1]{SinghSwanson}, and hence, $S$ is not a Bernstein algebra by \Cref{prop: finite ass}.
\end{example}
	
\bibliographystyle{skalpha}
\bibliography{refs}

\newcommand{\etalchar}[1]{$^{#1}$}
\providecommand{\bysame}{\leavevmode\hbox to3em{\hrulefill}\thinspace}
\providecommand{\MR}{\relax\ifhmode\unskip\space\fi MR}
\providecommand{\MRhref}[2]{%
  \href{http://www.ams.org/mathscinet-getitem?mr=#1}{#2}
}
\providecommand{\href}[2]{#2}
\begin{thebibliography}{R{\v{S}}VdB19b}

\bibitem[AE05]{AE}
{\sc I.~M. Aberbach and F.~Enescu}: \emph{The structure of {$F$}-pure rings},
  Math.\ Z. \textbf{250} (2005), no.~4, 791--806.

\bibitem[AMHJ{\etalchar{+}}21]{square}
{\sc J.~\`Alvarez~Montaner, D.~J. Hern\'andez, J.~Jeffries, L.~N\'u\~nez
  Betancourt, P.~Teixeira, and E.~E. Witt}: \emph{Bernstein--{S}ato functional
  equations, {$V$}-filtrations, and multiplier ideals of direct summands}, to
  appear in Communications in Contemporary Mathematics; available online at
  \url{https://doi.org/10.1142/S0219199721500838}.

\bibitem[{\`A}HNB17]{AMHNB}
{\sc J.~{{\`A}lvarez Montaner}, C.~Huneke, and L.~N{\'u}{\~n}ez-Betancourt}:
  \emph{{$D$}-modules, {B}ernstein--{S}ato polynomials and {$F$}-invariants of
  direct summands}, Adv. Math. \textbf{321} (2017), 298--325.

\bibitem[{\`A}MJNB21]{SurveyBS}
{\sc J.~{\`A}lvarez~Montaner, J.~Jeffries, and L.~N{\'u}{\~{n}}ez-Betancourt}:
  \emph{Bernstein-sato polynomials in commutative algebra}, Commutative
  Algebra: Expository Papers Dedicated to David Eisenbud on the Occasion of his
  75th Birthday (I.~Peeva, ed.), Springer International Publishing, Cham, 2021,
  pp.~1--76.

\bibitem[Ati56]{Atiyah}
{\sc M.~Atiyah}: \emph{On the {K}rull-{S}chmidt theorem with application to
  sheaves}, Bull. Soc. Math. France \textbf{84} (1956), 307--317.

\bibitem[Bav09]{Bavula}
{\sc V.~V. Bavula}: \emph{Dimension, multiplicity, holonomic modules, and an
  analogue of the inequality of {B}ernstein for rings of differential operators
  in prime characteristic}, Represent. Theory \textbf{13} (2009), 182--227.

\bibitem[Ber72]{Ber72}
{\sc I.~N. Bern\v{s}te\u{\i}n}: \emph{Analytic continuation of generalized
  functions with respect to a parameter}, Funkcional. Anal. i Prilo\v zen.
  \textbf{6} (1972), no.~4, 26--40.

\bibitem[Bj{\"{o}}79]{Bjork79}
{\sc J.-E. Bj{\"{o}}rk}: \emph{Rings of differential operators}, North-Holland
  Mathematical Library, vol.~21, North-Holland, Amsterdam, 1979.

\bibitem[Bol13]{Boldini}
{\sc R.~Boldini}: \emph{Critical cones of characteristic varieties}, Trans.
  Amer. Math. Soc. \textbf{365} (2013), no.~1, 143--160.

\bibitem[Bou07]{bourbaki}
{\sc N.~Bourbaki}: \emph{Alg{\`e}bre commutative\textup: Chapitres 1 {\`a} 4},
  Springer, Berlin, 2007.

\bibitem[BJNB19]{BJNB}
{\sc H.~Brenner, J.~Jeffries, and L.~N{\'u}{\~n}ez-Betancourt}:
  \emph{Quantifying singularities with differential operators}, Adv. Math.
  \textbf{358} (2019), 106843.

\bibitem[Cou95]{Coutinho}
{\sc S.~C. Coutinho}: \emph{A {P}rimer of {A}lgebraic {$D$}-{M}odules}, London
  Mathematical Society Student Texts, vol.~33, Cambridge University Press,
  Cambridge, 1995.

\bibitem[DDSG{\etalchar{+}}18]{SurveySP}
{\sc H.~Dao, A.~De~Stefani, E.~Grifo, C.~Huneke, and
  L.~{N{\'u}{\~n}ez-Betancourt}}: \emph{Symbolic powers of ideals},
  Singularities and {F}oliations. {G}eometry, {T}opology and {A}pplications,
  Springer Proc.\ Math.\ Stat., vol. 222, Springer, Cham, 2018, pp.~387--432.

\bibitem[DQ20]{DQ}
{\sc H.~Dao and P.~H. Quy}: \emph{On the associated primes of local
  cohomology}, Nagoya Math. J. \textbf{237} (2020), 1--9.

\bibitem[DNnB22]{DNB}
{\sc D.~Duarte and L.~N\'{u}\~{n}ez Betancourt}: \emph{Nash blowups in prime
  characteristic}, Rev. Mat. Iberoam. \textbf{38} (2022), no.~1, 257--267.

\bibitem[Gab81]{Gabber}
{\sc O.~Gabber}: \emph{The integrability of the characteristic variety}, Amer.
  J. Math. \textbf{103} (1981), no.~3, 445--468.

\bibitem[HH89]{HoHuStrong}
{\sc M.~Hochster and C.~Huneke}: \emph{Tight closure and strong
  {$F$}-regularity}, M\'em. Soc. Math. France \textup(N.S.\textup) \textbf{38}
  (1989), 119--133, Colloque en l'honneur de Pierre Samuel (Orsay, 1987).

\bibitem[HNB17]{HNB}
{\sc M.~Hochster and L.~N{\'u}{\~n}ez-Betancourt}: \emph{Support of local
  cohomology modules over hypersurfaces and rings with {FFRT}}, Math. Res.
  Lett. \textbf{24} (2017), no.~2, 401--420.

\bibitem[HR76]{HR1976}
{\sc M.~Hochster and J.~L. Roberts}: \emph{The purity of the {F}robenius and
  local cohomology}, Adv. Math. \textbf{21} (1976), no.~2, 117--172.

\bibitem[Kan77]{Kantor}
{\sc J.-M. Kantor}: \emph{Formes et op\'erateurs diff{\'e}rentiels sur les
  espaces analytiques complexes}, Bull. Soc. Math. France M{\'e}m. \textbf{53}
  (1977), 5--80.

\bibitem[Kas75]{KashiwaraRham}
{\sc M.~Kashiwara}: \emph{On the maximally overdetermined system of linear
  differential equations, {I}}, Publ. Res. Inst. Math. Sci. \textbf{10} (1975),
  no.~2, 563--579.

\bibitem[Kas80]{MKRH0}
{\sc M.~Kashiwara}: \emph{Faisceaux constructibles et syst{\`e}mes
  holon\^{o}mes d{\'e}quations aux d{\'e}riv\'{e}es partielles lin{\'e}aires
  \`a points singuliers r{\'e}guliers}, S{\'e}minaire {G}oulaouic-{S}chwartz,
  1979--1980 \textup({F}rench\textup), Exp.\ No.\ 19, {\'E}cole Polytech.,
  Palaiseau, 1980, pp.~1--6.

\bibitem[Kas84]{MKRH1}
{\sc M.~Kashiwara}: \emph{The {R}iemann--{H}ilbert problem for holonomic
  systems}, Publ. Res. Inst. Math. Sci. \textbf{20} (1984), no.~2, 319--365.

\bibitem[KLZ12]{TwoExamples}
{\sc M.~Katzman, G.~Lyubeznik, and W.~Zhang}: \emph{Two interesting examples of
  {$D$}-modules in characteristic {$p>0$}}, Bull. Lond. Math. Soc. \textbf{44}
  (2012), no.~6, 1116--1122.

\bibitem[Kem02]{Kemper}
{\sc G.~Kemper}: \emph{Loci in quotients by finite groups, pointwise
  stabilizers and the {B}uchsbaum property}, J. Reine Angew. Math. \textbf{547}
  (2002), 69--96.

\bibitem[Kun02]{KunzADC}
{\sc E.~Kunz}: \emph{Algebraic differential calculus}, unpublished manuscript,
  available at
  \url{https://www.uni-regensburg.de/Fakultaeten/nat_Fak_I/kunz/index.html},
  2002.

\bibitem[Li12]{Li}
{\sc H.~Li}: \emph{On monoid graded local rings}, J. Pure Appl. Algebra
  \textbf{216} (2012), no.~12, 2697--2708.

\bibitem[Los17]{Losev}
{\sc I.~Losev}: \emph{Bernstein inequality and holonomic modules}, Adv. Math.
  \textbf{308} (2017), 941--963, with an appendix by Losev and Pavel Etingof.

\bibitem[Lyu93]{LyuDMod}
{\sc G.~Lyubeznik}: \emph{Finiteness properties of local cohomology modules
  \textup(an application of {$D$}-modules to commutative algebra\textup)},
  Invent. Math. \textbf{113} (1993), no.~1, 41--55.

\bibitem[Mal79]{Malgrange79}
{\sc B.~Malgrange}: \emph{L'involutivit\'{e} des caract\'{e}ristiques des
  syst\`emes diff\'{e}rentiels et microdiff\'{e}rentiels}, S\'{e}minaire
  {B}ourbaki, 30e ann\'{e}e \textup(1977/78\textup), Lecture Notes in Math.,
  vol. 710, Springer, Berlin, 1979, pp.~277--289.

\bibitem[Mal21]{Mallory}
{\sc D.~Mallory}: \emph{Bigness of the tangent bundle of del {P}ezzo surfaces
  and {$D$}-simplicity}, Algebra Number Theory \textbf{15} (2021), no.~8,
  2019--2036.

\bibitem[MR87]{McC_Rob}
{\sc J.~C. McConnell and J.~C. Robson}: \emph{Noncommutative {N}oetherian
  rings}, Pure and Applied Mathematics (New York), John Wiley \& Sons, Ltd.,
  Chichester, 1987, with the cooperation of L.\ W.\ Small.

\bibitem[Meb80]{ZMRH0}
{\sc Z.~Mebkhout}: \emph{Sur le probl{\`e}me de {H}ilbert-{R}iemann}, Complex
  analysis, microlocal calculus and relativistic quantum theory
  \textup({P}roc.\ {I}nternat.\ {C}olloq., {C}entre {P}hys., {L}es {H}ouches,
  1979\textup), Lecture Notes in Phys., vol. 126, Springer, Berlin-New York,
  1980, pp.~90--110.

\bibitem[Meb84a]{ZMRH2}
{\sc Z.~Mebkhout}: \emph{Une autre \'{e}quivalence de cat\'{e}gories},
  Compositio Math. \textbf{51} (1984), no.~1, 63--88.

\bibitem[Meb84b]{ZMRH1}
{\sc Z.~Mebkhout}: \emph{Une {\'e}quivalence de cat\'{e}gories}, Compositio
  Math. \textbf{51} (1984), no.~1, 51--62.

\bibitem[Mil86]{Milicic}
{\sc D.~Mili\v{c}i\'c}: \emph{Lectures on algebraic theory of {$D$}-modules},
  University of Utah course notes, available at
  \url{https://www.math.utah.edu/~milicic/Eprints/dmodules.pdf}, 1986.

\bibitem[MM18]{Moncada}
{\sc L.~N. Moncada~Morales}: \emph{Rings of differential operators}, Master's
  thesis, Centro de Investigaci{\'o}n en Matem{\'a}ticas, Guanajuato, Mexico,
  2018.

\bibitem[NVO82]{GradedRingTheory}
{\sc C.~N{\u{a}}st{\u{a}}sescu and F.~Van~Oystaeyen}: \emph{Graded ring
  theory}, North-Holland Mathematical Library, vol.~28, North-Holland,
  Amsterdam, 1982.

\bibitem[NS02]{neusel}
{\sc M.~Neusel and L.~Smith}: \emph{Invariant theory of finite groups},
  American Mathematical Society, Providence, R.I, 2002.

\bibitem[NB12]{NBDS}
{\sc L.~N{\'u}{\~n}ez-Betancourt}: \emph{Local cohomology properties of direct
  summands}, J. Pure Appl. Algebra \textbf{216} (2012), no.~10, 2137--2140.

\bibitem[PT18]{PT}
{\sc T.~Polstra and K.~Tucker}: \emph{{$F$}-signature and {H}ilbert--{K}unz
  multiplicity: a combined approach and comparison}, Algebra Number Theory
  \textbf{12} (2018), no.~1, 61--97.

\bibitem[R{\v{S}}VdB19a]{RSV1}
{\sc T.~Raedschelders, {\v{S}}.~{\v{S}}penko, and M.~Van~den Bergh}: \emph{The
  {F}robenius morphism in invariant theory}, Adv. Math. \textbf{348} (2019),
  183--254.

\bibitem[R{\v{S}}VdB19b]{RSV2}
{\sc T.~Raedschelders, {\v{S}}.~{\v{S}}penko, and M.~Van~den Bergh}: \emph{The
  {F}robenius morphism in invariant theory {II}}, preprint,
  \href{https://arxiv.org/abs/1901.10956}{arXiv:1901.10956 [math.AG]}, 2019.

\bibitem[SKK73]{SKK}
{\sc M.~Sato, T.~Kawai, and M.~Kashiwara}: \emph{Microfunctions and
  pseudo-differential equations}, Hyperfunctions and pseudo-differential
  equations \textup({P}roc.\ {C}onf., {K}atata, 1971; dedicated to the memory
  of {A}ndr\'{e} {M}artineau\textup), Lecture Notes in Math., vol. 287,
  Springer, Berlin, 1973, pp.~265--529.

\bibitem[SS72]{SatoPoly}
{\sc M.~Sato and T.~Shintani}: \emph{On zeta functions associated with
  prehomogeneous vector spaces}, Proc. Nat. Acad. Sci. U.S.A. \textbf{69}
  (1972), 1081--1082.

\bibitem[Sch95]{Schwarz}
{\sc G.~W. Schwarz}: \emph{Lifting differential operators from orbit spaces},
  Ann. Sci. \'Ecole Norm. Sup. \textup(4\textup) \textbf{28} (1995), no.~3,
  253--305.

\bibitem[ST54]{ShephardTodd}
{\sc G.~C. Shephard and J.~A. Todd}: \emph{Finite unitary reflection groups},
  Canad. J. Math. \textbf{6} (1954), 274--304.

\bibitem[SS04]{SinghSwanson}
{\sc A.~K. Singh and I.~Swanson}: \emph{Associated primes of local cohomology
  modules and of {F}robenius powers}, Int. Math. Res. Not. \textbf{2004}
  (2004), no.~33, 1703--1733.

\bibitem[Smi01]{GregSmith}
{\sc G.~G. Smith}: \emph{Irreducible components of characteristic varieties},
  J. Pure Appl. Algebra \textbf{165} (2001), no.~3, 291--306.

\bibitem[SVdB97]{SVDB}
{\sc K.~E. Smith and M.~Van~den Bergh}: \emph{Simplicity of rings of
  differential operators in prime characteristic}, Proc. London Math. Soc.
  \textup(3\textup) \textbf{75} (1997), no.~1, 32--62.

\bibitem[Smi87]{SmithSP}
{\sc S.~P. Smith}: \emph{The global homological dimension of the ring of
  differential operators on a nonsingular variety over a field of positive
  characteristic}, J. Algebra \textbf{107} (1987), no.~1, 98--105.

\bibitem[TT08]{TT}
{\sc S.~Takagi and R.~Takahashi}: \emph{{$D$}-modules over rings with finite
  {$F$}-representation type}, Math. Res. Lett. \textbf{15} (2008), no.~3,
  563--581.

\bibitem[Tra06]{Traves}
{\sc W.~N. Traves}: \emph{Differential operators on orbifolds}, J. Symbolic
  Comput. \textbf{41} (2006), no.~12, 1295--1308.

\bibitem[Tuc12]{Tucker12}
{\sc K.~Tucker}: \emph{{$F$}-signature exists}, Invent. Math. \textbf{190}
  (2012), no.~3, 743--765.

\bibitem[vdE85]{VdEdRham}
{\sc A.~van~den Essen}: \emph{The cokernel of the operator {$\partial/\partial
  x_n$} acting on a {${D}_n$}-module. {II}}, Compositio Math. \textbf{56}
  (1985), no.~2, 259--269.

\bibitem[vdE86]{VdE86}
{\sc A.~van~den Essen}: \emph{Modules with regular singularities over filtered
  rings}, Publ. Res. Inst. Math. Sci. \textbf{22} (1986), no.~5, 849--887.

\bibitem[Yao06]{Yao06}
{\sc Y.~Yao}: \emph{Observations on the {$F$}-signature of local rings of
  characteristic {$p$}}, J. Algebra \textbf{299} (2006), no.~1, 198--218.

\bibitem[Yek92]{yekutieli.explicit_construction}
{\sc A.~Yekutieli}: \emph{An explicit construction of the {G}rothendieck
  residue complex}, Ast\'{e}risque \textbf{208} (1992), 3--115, with an
  appendix by Pramathanath Sastry.

\end{thebibliography}
\end{document}